\documentclass[reqno,12pt]{amsart}

\usepackage{geometry,amsmath,amsfonts,amssymb,amsthm,amscd,mathrsfs,graphicx,enumerate,multicol,wrapfig,subfigure,hyperref,stmaryrd, accents}
\usepackage[all]{xy}
\usepackage{setspace}
\usepackage[T1]{fontenc}

\numberwithin{equation}{subsection}

\newcommand{\ra}{\rightarrow}

\newcommand{\lra}{\longrightarrow}

\newcommand{\p}{\prime}
\newcommand{\pt}{\partial}

\newcommand{\al}{\alpha}

\newcommand{\om}{\omega}

\newcommand{\lam}{\lambda}

\newcommand{\q}{\theta}

\newcommand{\dt}{\delta}

\newcommand{\Zbb}{\mathbb{Z}}

\newcommand{\Cbb}{\mathbb{C}}

\theoremstyle{plain} 

\newtheorem{THM}{Theorem}[section]
\newtheorem{DEF}[THM]{Definition}
\newtheorem{EX}[THM]{Example}

\newtheorem{QUE}[THM]{Question}

\newtheorem{PROP}[THM]{Proposition}
\newtheorem{LEM}[THM]{Lemma}
\newtheorem{COR}[THM]{Corollary}
\newtheorem{REM}[THM]{Remark}

\newcommand{\bt}{\bullet}

\newcommand{\Aut}{\mathrm{Aut}}

\newcommand{\Ac}{\mathcal{A}}
\newcommand{\Bc}{\mathcal{B}}

\newcommand{\img}{\mathrm{im}}

\newcommand{\Uc}{\mathcal{U}}

\newcommand{\Oc}{\mathcal{O}}

\newcommand{\Jc}{\mathcal{J}}

\newcommand{\Cc}{\mathcal{C}}
\newcommand{\Gc}{\mathcal{G}}

\newcommand{\Hc}{\mathcal{H}}
\newcommand{\Fc}{\mathcal{F}}

\newcommand{\Xfr}{\mathfrak{X}}

\newcommand{\Ufr}{\mathfrak{U}}

\newcommand{\Mfr}{\mathfrak{M}}

\newcommand{\Scl}{\mathcal{S}}
\newcommand{\Xc}{\mathcal{X}}
\newcommand{\Mcl}{\mathcal{M}}

\usepackage{xcolor}
\definecolor{airforceblue}{rgb}{0.36, 0.54, 0.66}
\definecolor{burgundy}{rgb}{0.5, 0.0, 0.13}
\definecolor{majorelleblue}{rgb}{0.38, 0.31, 0.86}
\definecolor{darkblue}{rgb}{0.0, 0.0, 0.55}

\hypersetup{urlcolor=burgundy, linkcolor=burgundy,
citecolor=darkblue,
colorlinks=true}

\title[Sh. of AQ Norm. Series and Smflds.]
{Sheaves of AQ Normal Series and Supermanifolds
\\}
\author{\small Kowshik Bettadapura}
\date{}

\begin{document}
\maketitle

\begin{abstract} 
On a group $G$, a filtration by normal subgroups is referred to as a normal series. If subsequent quotients are abelian, the filtration is referred to as an \emph{abelian-quotient normal series}, or `AQ normal series' for short. In this article we consider `sheaves of AQ normal series'. From a given AQ normal series satisfying an additional hypothesis we derive a complex whose first cohomology obstructs the resolution of an `integration problem'. These constructs are then applied to the classification of supermanifolds modelled on $(X, T^*_{X, -})$, where $X$ is a complex manifold and $T^*_{X, -}$ is a holomorphic vector bundle. We are lead to the notion of an `obstruction complex' associated to a model $(X, T^*_{X, -})$ whose cohomology is referred to as `obstruction cohomology'. We deduce a number of interesting consequences of a vanishing first obstruction cohomology. Among the more interesting consequences are its relation to projectability of supermanifolds and a `Batchelor-type' theorem: \emph{if the obstruction cohomology of a `good' model $(X, T^*_{X, -})$ vanishes, then any supermanifold modelled on $(X, T^*_{X, -})$ will be split}.
\end{abstract}

\setcounter{tocdepth}{1}
\tableofcontents

\onehalfspacing

\section{Introduction}

\subsection{Motivation}
A pair $(X, T^*_{X, -})$ comprising a complex manifold $X$ and holomorphic vector bundle $T_{X, -}^*$ is referred to as a \emph{model}. It is the data required to define the notion of a `supermanifold'. Hence a supermanifold might be viewed as a structure associated to a model $(X, T^*_{X,-})$. In this article we are interesting in studying models $(X, T^*_{X, -})$ by reference to the supermanifolds modelled on them, in the hopes of gaining a greater understanding of supermanifolds in general.

\subsection{Sheaves of AQ Normal Series}
We begin from a general setting: `sheaves of AQ normal series' $\Hc$ on a topological space $X$. Associated to any $\Hc$ satisfying a centrality hypothesis we obtain a complex of vector spaces referred to as the \emph{linearisation complex} associated to $\Hc$. This construction allows for the formulation of a problem which we refer to as an `integration problem'. Loosely put, it concerns the problem of `integrating' elements in certain abelian sheaf cohomology groups to torsors which, for our purposes, are elements in degree $1$, non-abelian sheaf cohomology sets. This general setting is then applied to Green's classification of supermanifolds in \cite{GREEN}. We interpret Green's constructions in \cite{GREEN} as a sheaf of AQ normal series associated to any model $(X, T^*_{X, -})$. This allows us to then readily recover the main results in \cite{BETTOBSTHICK} as an application of our earlier, more general study. 

\subsection{Obstruction Cohomology}
As mentioned above, to any model $(X, T^*_{X, -})$ we obtain a sheaf of AQ normal series. It defines a complex of vector spaces whose cohomology is referred to as the \emph{obstruction cohomology} associated to $(X, T^*_{X, -})$. We are primarily interested in its degree $1$ component. In \cite{BETTOBSTHICK} a general classification of thickenings was proposed. Thickenings were classified as: \emph{supermanifolds, pseudo-supermanifolds or obstructed thickenings}. Then, almost by construction:
\begin{enumerate}[]
	\item if the $1$-obstruction cohomology of a model $(X, T^*_{X, -})$ vanishes, there will not exist any pseudo-supermanifolds modelled on $(X, T^*_{X, -})$.
\end{enumerate}
In this way we relate the obstruction cohomology of a model with the classification of thickenings from \cite{BETTOBSTHICK}. Indeed, with the vanishing of the obstruction cohomology we identify conditions under which the results in \cite{BETTOBSTHICK} are markedly improved (c.f., Theorem \ref{rfuihf3hf98h3f89h398fh3}). Therefore, determining conditions under which the obstruction cohomology of a model vanishes becomes a meaningful endeavour. We will see that when the model $(X, T^*_{X, -})$ is `good', in the sense of \cite{BETTHIGHOBS}, the vanishing of its obstruction cohomology is synonymous with the injectivity of a certain map (Theorem \ref{rfg7gfi3hfiuwhfioh}), leading then to a nice simplification when $X$ is a Riemann surface in Corollary \ref{fjcdkbvjdbvjhbvrbvbi}.

\subsection{Classification of Models}
In \cite{BETTHIGHOBS} the notion of a `good model' was introduced. We continue this classification-of-models by introducing notions of `projectablility' and `splitness' in analogy with supermanifolds. A model $(X, T^*_{X, -})$ is said to be \emph{projectable} resp., \emph{split} if every supermanifold modelled on $(X, T^*_{X, -})$ is projectable resp., split. Our ultimate applications in this article concerns conditions under which a model $(X, T^*_{X, -})$ will be projectable or split. Results of the latter kind are thought of as `Batchelor-type' theorems as they generalise the classical Batchelor's theorem in the smooth setting to the complex-analytic setting. We find:
\begin{enumerate}[$\bt$]
	\item (Theorem \ref{fvndvkbvubvknenvkno}) if the $1$-obstruction cohomology of a model vanishes, then the model is projectable;
	\item (Theorem \ref{fnvburbvuibuinenfneoifeoo}) if the $1$-obstruction cohomology of a \emph{good} model vanishes, then the model is split.
\end{enumerate}
Hence we see that the $1$-obstruction cohomology plays an integral role in the classification of models. As an application of Theorem \ref{fvndvkbvubvknenvkno} on the projectability of models and Donagi and Witten's result in \cite{DW1} on the non-projectability of the supermoduli space of curves, we deduce in Example \ref{dkvndbvjkbukvbuevevnkek} that the obstruction cohomology of the modelling data for supermoduli space cannot vanish. 
\\\\
En route to proving Theorem \ref{fvndvkbvubvknenvkno} and \ref{fnvburbvuibuinenfneoifeoo}, we will need a vanishing result mentioned by Donagi and Witten in \cite{DW1}. An exposition of these results are presented and, subsequently, we derive a stronger vanishing statement in Theorem \ref{dkcdnlnvdvbebvukenvl}. The results culminating in Theorem \ref{dkcdnlnvdvbebvukenvl} is referred to as the Vanish-Lift-Vanish principle and Appendix \ref{fklnvkbvjrbvnfkemlkmel} is devoted to a proof of this principle.  As a further consequence of this principle we obtain in Theorem \ref{rfh48gh4hg4040gj490j} a condition under which Berezin's lift, a brief exposition of which is given in Section \ref{fcbeuyvyegv8g79g93h93}, will be unique. This may serve as a result of independent interest.

\subsection{Outlook for Future Work}
This article is largely theoretical and illuminating examples are not so forthcoming. What ought to be clear however is that the obstruction cohomology of a model $(X, T^*_{X, -})$ carries important information pertaining to the classification of thickenings and supermanifolds. Following Berezin's notion of `splitness' (referred to as `simple' in \cite{BER}), Green in \cite{GREEN} provided a set-theoretic classification and Onishchik in \cite{ONISHCLASS} considered the moduli problem for this classification. Where the deformation theory is concerned, Vaintrob in \cite{VAIN} gave a general account, largely independent of the work by Green and Onishchik, i.e., in that `splitness', or lack thereof, was not explicitly considered. More recently Donagi and Witten in \cite{DW2} commented on the difficulties behind deforming non-split supermanifolds. With the material presented in this article we hope, in future works, to apply it to interesting deformation problems involving supermanifolds and, more generally, to variational problems involving superfields in physics (thought of as morphisms between supermanifolds).

\section{Sheaves of AQ-Normal Series}

\subsection{Definitions}
The terminology appearing in this section is largely standard and can be found, e.g., in texts such as \cite[p. 107]{HUNG}.

\subsubsection{Normal Series}
Let $G$ be a group and $H\leq G$ a subgroup. As sets, we can form the quotient $G/H$ which is the set of $H$-orbits. If $H$ is a normal subgroup, then the orbit set $G/H$ will admit a natural group structure. The condition of normality is generally not transitive. That is, if $H^\p\lhd H$ and $H\lhd G$, it need \emph{not} be the case that $H^\p\lhd G$. 

\begin{DEF}
\emph{Let $G$ be a group. A finite collection of subgroups $\Hc = (H_j)_{j = 0, \ldots,N}$ of $G$ satisfying:
\begin{enumerate}[(i)]
	\item $H_N = \{1\}$;
	\item $H_{j+1}\lhd H_j$ for all $0\leq j< N$ and;
	\item $H_0= G$
\end{enumerate}
is called a \emph{subnormal series for $G$} of \emph{length $N$}.}
\end{DEF}

\begin{DEF}
\emph{Let $G$ be a group. A subnormal series $\Hc = (H_j)$ for $G$ is said to be a \emph{normal} if each $H_j\in \Hc$ is a normal subgroup of $G$.}
\end{DEF}

\noindent
In what follows we will reference the normal series $\Hc =(H_j)_{j\geq 0}$ with $G = H_0$ understood. A very general result which we will make use of a number of times in this article is the following:

\begin{LEM}\label{fbciuebubennekwww}
Let $\Hc = \{H^\p, H, G\}$ be a normal series. Then there exists an isomorphism of groups
\[
H/H^\p \stackrel{\sim}{\lra} \ker\{ G/H^\p\ra G/H\}.
\]
\end{LEM}

\begin{proof}
Since $\{H^\p, H, G\}$ is a normal series we know that $H^\p\lhd H$, $H\lhd G$ and $H^\p\lhd G$. Thus we can form the following commutative diagram with exact rows in the category of groups,
\begin{align}
\xymatrix{
\{e\}\ar[r] & \ar[d] H^\p \ar[r] & G\ar@{=}[d]\ar[r] & G/H^\p\ar[d]\ar[r] & \{e\}
\\
\{e\}\ar[r] & H \ar[r] & G\ar[r] & G/H \ar[r] & \{e\}.
}
\label{rhf784f7hf9830fj039j03}
\end{align}
The induced map $G/H^\p\ra G/H$ is surjective. Let $K$ be its kernel. By exactness of the rows in \eqref{rhf784f7hf9830fj039j03} we get a surjective homomorphism $H\ra K$ with kernel $H^\p$. Hence by the First Isomorphism Theorem for groups, $H/H^\p$ and $K$ are isomorphic.
\end{proof}

\noindent
To reiterate Lemma \ref{fbciuebubennekwww}: \emph{associated to any normal series $\{H^\p\lhd H\lhd G\}$ is a short exact sequence of groups}:
\begin{align}
\{1\}\lra H/H^\p\lra G/H^\p \lra G/H \lra\{1\}.
\label{ru4hf8hf09jf90j30}
\end{align}
In what follows we look at normal series whose subsequent quotients are Abelain groups.

\subsubsection{AQ Normal Series}

\begin{DEF}\label{fvjdiovnfvnjkvnkjrnkvj}
\emph{Let $G$ be a group and $H\lhd G$ be a normal subgroup. If the quotient $G/H$ is \emph{abelian}, then $H$ is said to be an \emph{abelian-quotient subgroup}. We will abbreviate by referring to $H\lhd G$ as an \emph{AQ subgroup}.}
\end{DEF}

\noindent
With Definition \ref{fvjdiovnfvnjkvnkjrnkvj} we can form the notion of an `AQ normal series'

\begin{DEF}
\emph{A normal series $\Hc = (H_j)_{j\geq0}$ is said to be an \emph{AQ normal series} if $H_{j+1}$ is an AQ subgroup of $H_j$ for each $j$.}
\end{DEF}

\noindent
Note that if $\Hc = (H_j)$ is a normal series, then $H_{j^\p}$ will be a normal subgroup of $H_j$ for any $j^\p\geq j$ and so we can form the quotient $H_j/H_{j^\p}$. This observation leads to the notion of an `AQ degree' of a normal series.

\begin{DEF}\label{rfg4fg87hf89h38f30}
\emph{Let $\Hc = (H_j)_{j = 0, \ldots, N}$ be a normal series. We say $\Hc$ has \emph{AQ degree $d$} if, for each $j$, the quotients  $H_j/H_{j+d}, H_j/H_{j+d-1}, \ldots, H_j/H_{j+1}$ are abelian.}
\end{DEF}

\begin{REM}
\emph{Any AQ normal series will have AQ degree $1$ and any normal series of AQ degree $d$ will define a series with AQ degree $d^\p$ for any $d^\p\leq d$. In particular, any normal series with AQ degree $d>1$ will be an AQ normal series.}
\end{REM}

\subsection{Sheaves of AQ Central Series}
Let $X$ be a topological space. We will consider on $X$ a sheaf of AQ normal series $\Hc = (\Hc_j)_{j\geq 0}$. This means, for each $j$:
\begin{enumerate}[$\bt$]
	\item $\Hc_j$ is a sheaf of groups;
	\item the inclusion $\Hc_{j+1}\subset \Hc_j$ realises $\Hc_{j+1}$ as a sheaf of normal subgroups of $\Hc_j$;
	\item the inclusion $\Hc_j\subset \Hc_0 =  \Gc$ realises $\Hc_j$ as a sheaf of normal subgroups of $\Gc$;
	\item the quotient $\Hc_j/\Hc_{j+1}$ is a sheaf of abelian groups for all $j$.
\end{enumerate}
The notion of AQ degree in Definition \ref{rfg4fg87hf89h38f30} adapts straightforwardly to sheaves of AQ series $\Hc =(\Hc_j)$ here. By \eqref{ru4hf8hf09jf90j30} there will be a short exact sequence of sheaves of groups for each $j> k$,
\begin{align}
\{e\}
\lra
\Hc_j/\Hc_{j+1}
\lra 
\Hc_{k}/\Hc_{j+1}
\lra 
\Hc_{k}/\Hc_j
\lra
\{e\}.
\label{rhf973hf93hf8jf0j30}
\end{align}
Since $\Hc$ is an AQ normal series the quotient $\Hc_j/\Hc_{j+1}$ is abelian. We will say $\Hc = (\Hc_j)$ is \emph{central} if, for each $j>0$, there exists some $k< j$ such that the sequence in \eqref{rhf973hf93hf8jf0j30} is central, i.e., that the inclusion $\Hc_j/\Hc_{j+1}$ is contained in the centre of $\Hc_{k}/\Hc_{j+1}$ for {some} $k$. 

\begin{DEF}\label{rfh894fh894hf0jf093j0}
\emph{A sheaf of AQ normal series which is central will be referred to as an \emph{AQ central series}.}
\end{DEF}

\noindent
We can now derive the following structure.

\begin{PROP}\label{rfbyuef8f9h39h3}
Let $X$ be a topological space and $\Hc =(\Hc_j)_{j\geq0}$ a sheaf of AQ central series on $X$ with successive quotients denoted $\Ac_j = \Hc_j/\Hc_{j+1}$. Then for each $j$ there exists a complex of pointed sets
\[
H^0\big(X, \Ac_j\big) \stackrel{\pt^1_{j+1}}{\lra} H^1\big(X, \Ac_{j+1}\big) \stackrel{\pt^2_{j+2}}{\lra} H^2\big(X, \Ac_{j+2}\big),
\]
i.e., that the composition $\pt^2_{j+1}\pt^1_{j+1}$ is trivial for each $j$.
\end{PROP}

\begin{proof}
By definition of an AQ central series in Definition \ref{rfh894fh894hf0jf093j0} there will exist, for each $j$, some $k< j$ such that the following short exact sequence of sheaves is central,
\begin{align}
\{e\} \lra \Ac_{j} \lra \Hc_k/\Hc_{j+1} \lra \Hc_k/\Hc_{j} \lra \{e\}.
\label{rfh4hf4h9f83fj03}
\end{align}
Grothendieck in \cite{GROTHNONAB} observed that a short exact sequence of sheaves of groups will give rise to a long-exact sequence on \v Cech cohomology (in degrees zero and one). It extends to a map in degree two when  the sequence is central and so we can apply this observation to \eqref{rfh4hf4h9f83fj03}, yielding:
\begin{align}
\xymatrix{
\{1\}\ar[r] &H^0\big(X, \Ac_{j}\big) \ar[r]^{\iota^0_{j+1}} &  H^0\big(X, \Hc_k/\Hc_{j+1}\big)\ar[r] & H^0\big(X, \Hc_k/\Hc_{j}\big)\ar[d]^{\dt^1_{j}}
\\
&\ar[d]_{\dt^2_{j}} \mbox{\v H}^1\big(X, \Hc_k/\Hc_{j}\big) &\ar[l] \mbox{\v H}^1\big(X, \Hc_k/\Hc_{j+1}\big) &  \ar[l]_{\iota^1_{j+1}} H^1\big(X, \Ac_{j}\big)
\\
&H^2\big(X, \Ac_{j}\big)& & 
}
\label{fjnvfvjnvkjndkjvndk}
\end{align}
Thus we obtain the following maps comparing the cohomology of the abelian sheaves $\Ac_j$ for differing $j$,
\begin{align}
\xymatrix{
& H^0\big(X, \Hc_k/\Hc_{j+1}\big)\ar[dr]^{\dt^1_{j+1}} &
\\
H^0\big(X, \Ac_j\big) \ar[ur]^{\iota^0_{j+1}} \ar@{-->}[rr]_{\pt_{j+1}^1} & & H^1\big(X, \Ac_{j+1}\big)\ar@{-->}[dd]_{\pt^2_{j+2}} \ar[dr]^{\iota^1_{j+2}}
\\
& & &  \mbox{\v H}^1\big(X, \Hc_k/\Hc_{j+2}\big)\ar[dl]^{\dt^2_{j+2}}
\\
& & H^2\big(X, \Ac_{j+2}\big) &
}
\label{rf487fg873hfh89fh2j90j902j}
\end{align}
The composition $\pt^2_{j+2}\pt^1_{j+1}$ vanishes since the composition $\iota^1_{j+1}\dt^1_{j+1}$ is trivial (c.f., \eqref{fjnvfvjnvkjndkjvndk}). The proposition now follows.
\end{proof}

\subsubsection{As $\Zbb$-Graded abelian Sheaves}
We can reformulate Proposition \ref{rfbyuef8f9h39h3} in a more succinct manner as follows. To a sheaf of AQ central series $\Hc = (\Hc_j)$ on $X$ we can construct the $\Zbb$-graded, abelian sheaf $\Ac = \oplus_j \Ac_j$, where $\Ac_j = \Hc_j/\Hc_{j+1}$. 

\begin{DEF}\label{rburbuibuinenoienioen}
\emph{The sheaf $\Ac = \oplus_j \Ac_j$ is referred to as the \emph{linearisation of $\Hc$}.}
\end{DEF}

\noindent
Since cohomology commutes with direct sums, the $\Zbb$-grading on $\Ac$ gives a $\Zbb$-grading on the cohomology groups $H^n(X, \Ac)$. Proposition \ref{rfbyuef8f9h39h3} then says: \emph{there exists a complex $H^0\big(X, \Ac[-1]\big) \stackrel{\pt^1}{\ra} H^1\big(X, \Ac\big)\stackrel{\pt^2}{\ra} H^2\big(X, \Ac[1]\big)$ of set-theoretic maps with respect to the induced $\Zbb$-grading.} 

\begin{DEF}\label{g78f7f93hf83h803}
\emph{The complex $\big(H^0\big(X, \Ac[-1]\big) \stackrel{\pt^1}{\ra} H^1\big(X, \Ac\big)\stackrel{\pt^2}{\ra} H^2\big(X, \Ac[1]\big)\big)$ of pointed sets will be referred to as the \emph{primary complex of the AQ central series $\Hc$}. Its cohomology will be referred to as the \emph{primary cohomology} of $\Hc$ and will be denoted ${\bf H}^\bt_{\pt}(X, \Ac)$.}
\end{DEF}

\noindent
Any AQ normal series $\Hc$ with AQ degree $d> 1$ will be central. Hence Proposition \ref{rfbyuef8f9h39h3} will apply to any such a series $\Hc$. Now, a point to emphasise presently is that the primary complex of $\Hc$ in Definition \ref{g78f7f93hf83h803} is a complex of pointed sets. In the section to follow we will deduce linearity.

\section{Linearity}
\label{rhuybiunioeiofee}

\subsection{The Long Exact Sequence}
If $0\ra \Ac\ra \Bc\ra \Cc\ra0$ is a short exact sequence of abelian sheaves on $X$, then there exists a long exact sequence of complex vector spaces, a piece of which in degree $n$ is:
\begin{align}
\cdots\ra H^n(X, \Ac) \ra H^n(X, \Bc) \ra H^n(X, \Cc) \stackrel{\dt}{\ra} H^{n+1}(X, \Ac)\ra\cdots
\label{rfh89hf98hf83jf03jf}
\end{align}
Importantly the maps involved in \eqref{rfh89hf98hf83jf03jf}, including the boundary map $\dt$, are linear maps between vector spaces. Now if we are given a mapping $\pt: H^n(X, \Cc) \ra H^{n+1}(X, \Ac)$, defined at the level of sets, we can show $\pt$ is linear if we can find an extension $\Bc$ of $\Cc$ by $\Ac$ such that $\pt$ is the boundary map $\dt$ induced on cohomology in \eqref{rfh89hf98hf83jf03jf}. This is the strategy we employ to deduce linearity of the maps in the primary complex of a sheaf of AQ central series.

\subsection{Linearity of the Primary Complex}
In this section we will be concerned with the proof of the following.

\begin{THM}\label{fbvuyrbvybriuvoievpie}
Let $X$ be a topological space and $\Hc = (\Hc_j)$ a sheaf of AQ normal series with degree $2$. Then the primary  complex of $\Hc$ will be a complex of vector spaces.
\end{THM}

\begin{proof}
From Definition \ref{rfg4fg87hf89h38f30}, if $\Hc = (\Hc_j)$ has AQ-degree $2$ then for each $j$ the quotient
$\Hc_{j}/\Hc_{j+2}$ will be abelian. Adapting \eqref{ru4hf8hf09jf90j30} we have, for each $j$, an exact sequence of sheaves of abelian groups
\begin{align}
0\lra \Ac_{j+1} \lra \Hc_{j}/\Hc_{j+2} \lra \Ac_{j}\lra0
\label{rfh479hf93hf98j03020}
\end{align}
where $\Ac_j = \Hc_j/\Hc_{j+1}$. The sequence in \eqref{rfh479hf93hf98j03020} induces the following linear maps, which are the boundary maps in the long exact sequence on sheaf cohomology,
\begin{align}
H^n(X, \Ac_{j}) \stackrel{\widetilde\pt^{n+1}_{j+1}}{\lra} H^{n+1}(X, \Ac_{j+1}).
\label{rfh784fg873hf9hf389}
\end{align}
Thus there are linear maps,
\begin{align}
0\lra H^0(X, \Ac_{j}) \stackrel{\widetilde\pt^1_{j+1}}{\lra} H^1(X, \Ac_{j+1}) \stackrel{\widetilde\pt^2_{j+2}}{\lra} H^2(X, \Ac_{j+2}) \lra\cdots
\label{rhf894h89h9fj309303}
\end{align}
We now claim the following:
\begin{align}
\widetilde\pt^1_{j+1} = \pt^1_{j+1}
&&
\mbox{and}
&&
\widetilde\pt^1_{j+2} = \pt^1_{j+2}
\label{eh3hf983hf98h3}
\end{align}
where $\pt^1_{j+1}$ and $\pt^2_{j+2}$ are the maps in Proposition \ref{rfbyuef8f9h39h3}. To confirm \eqref{eh3hf983hf98h3} consider the following commutative diagram of sheaves with exact rows and columns,
\begin{align}
\xymatrix{
\Ac_{j+1}\ar@{=}[d] \ar[r] & \Hc_j/\Hc_{j+2} \ar[d] \ar[r] & \Ac_j\ar[d]
\\
\Ac_{j+1}\ar[r] & \Hc_k/\Hc_{j+2} \ar[d] \ar[r] & \Hc_k/\Hc_{j+1}\ar[d]
\\
&\Hc_k/\Hc_j\ar@{=}[r] &  \Hc_k/\Hc_j.
}
\label{jkdkjvbdvbbevkenle}
\end{align}
Evidently we obtain a commutative diagram on cohomology in degrees zero and one:
\begin{align*}
\xymatrix{
\ar[d]_{\iota^0_{j+1}} H^0\big(X, \Ac_j\big) \ar[r]^{\widetilde\pt^1_{j+1}} & H^1\big(X, \Ac_{j+1}\big)\ar@{=}[d]
\\
H^0\big(X, \Hc_k/\Hc_{j+1}\big) \ar[r]^{\dt^1_{j+1}} &H^1\big(X, \Ac_{j+1}\big)
}
&&
\mbox{and}
&&
\xymatrix{
\ar[d]_{\iota^2_{j+2}} H^1\big(X, \Ac_{j+1}\big) \ar[r]^{\widetilde\pt^2_{j+2}} & H^2\big(X, \Ac_{j+2}\big)\ar@{=}[d]
\\
\mbox{\v H}^1\big(X, \Hc_k/\Hc_{j+2}\big) \ar[r]^{\dt^2_{j+2}} &H^2\big(X, \Ac_{j+2}\big)
}
\end{align*}
Comparing with \eqref{rf487fg873hfh89fh2j90j902j} in Proposition \ref{rfbyuef8f9h39h3}, we see that commutativity of the above diagrams are precisely the equalities in \eqref{eh3hf983hf98h3}. The present theorem now follows.
\end{proof}

\noindent
Our applications in this article will ultimately rest on the interpretation of the \emph{first} primary cohomology ${\bf H}_\pt^1(X, \Ac)$ of a sheaf of AQ central series $\Hc$ and its relation to what we will term the `integration problem'. We conclude this section now with the observation that it is possible to extend the primary complex to obtain a larger complex of vector spaces.

\subsection{An Extension of the Primary Complex}
With the boundary maps in \eqref{rfh784fg873hf9hf389} the sequence in \eqref{rhf894h89h9fj309303} of course continues into the ellipses. In Theorem \ref{fbvuyrbvybriuvoievpie} the assumption that the AQ series have degree $2$ ensured only that the primary complex will be a complex of vector spaces. In what follows, the assumption on the AQ degree will be crucial in showing that the sequence in \eqref{rhf894h89h9fj309303} will also define a complex of vector spaces.

\begin{THM}\label{fvnrbvurnvnvkmvkel}
Let $X$ be a topological space and suppose $\Hc$ is an AQ series with degree $4$. Then the primary  complex of $\Hc$ extends to give a bounded linear complex $\big(\mathcal C^\bt(X, \Ac), \pt\big)$ which, in degree $n$, is:\footnote{for a $\Zbb$-graded module $\Fc = \oplus_j \Fc_j$, the shift $\Fc[\ell]$ is the module $\Fc$ with the grading: $(\Fc[\ell])_j = \Fc_{\ell+j}$.}
\begin{align*}
\Cc^n(X, \Ac)
= H^n(X, \Ac)
&&
\mbox{and}
&&
\pt : H^n\big(X, \Ac\big)\lra H^{n+1}\big(X, \Ac[1]\big).
\end{align*}
\end{THM}

\begin{proof}
The construction of the data in the alleged complex was given in the proof of Theorem \ref{fbvuyrbvybriuvoievpie}. It remains to show that the composition 
\[
H^n(X, \Ac_{j-1}) 
\stackrel{\pt^{n+1}_{j}}{\lra} 
H^{n+1}(X, \Ac_j)
\stackrel{\pt^{n+2}_{j+1}}{\lra}
H^{n+1}(X, \Ac_{j+1})
\]
will vanish. Here the assumption on the AQ degree will be essential. If $\Hc = (\Hc_j)$ is an AQ normal series with degree $4$, then the following quotients will be abelian:
\begin{align*}
\Hc_{j-1}/\Hc_{j+3}; \Hc_{j-1}/\Hc_{j+2}; \Hc_{j-1}/\Hc_{j+1}~\mbox{and}~\Hc_{j-1}/\Hc_{j}.
\end{align*}
Hence we can form their cohomology in any degree. Now observe that we have the following diagram, 
\begin{align*}
\xymatrix{
 & & \Ac_{j+1}\ar[d]
\\
 \Ac_{j+2}\ar[r] & \Hc_{j-1}/\Hc_{j+3}\ar[r]  &\Hc_{j-1}/\Hc_{j+2}\ar[d] &
\\ 
& \Ac_j \ar[r] & \Hc_{j-1}/\Hc_{j+1}\ar[r] & \Gc/\Hc_j
}
\end{align*}
On cohomology we therefore have:
\begin{align}
\xymatrix{
H^n\big(X, \Ac_j\big) \ar@{-->}[dr]_{\pt^{n+1}_{j+1}} \ar[r] & H^n\big(X, \Hc_{j-1}/\Hc_{j+1}\big)\ar[d]
\\
& H^{n+1}\big(X, \Ac_{j+1}\big) \ar[d] \ar@{-->}[dr]^{\pt^{n+2}_{j+2}} &
\\
&H^{n+1}\big(X, \Hc_{j-1}/\Hc_{j+2}\big) \ar[r] & H^{n+2}\big(X, \Ac_{j+2}\big)
}
\label{rfh784gf784hf893j903}
\end{align}
Since middle column in \eqref{rfh784gf784hf893j903} is exact it follows, by commutativity, that $\pt^{n+2}_{j+2}\pt^{n+1}_{j+1} = 0$. We therefore have a linear complex, as claimed. It is bounded since the terms in this complex, being cohomology groups of abelian sheaves on $X$, are trivial in negative degrees and degrees higher than the dimension of $X$.
\end{proof}

\subsection{Linearity on Global Sections}
In previous sections (see Theorem \ref{fbvuyrbvybriuvoievpie} and \ref{fvnrbvurnvnvkmvkel}) we deduced linearity of the boundary maps in the primary and extended complex of a sheaf of AQ series $\Hc$ by making an assumption on its degree. We note here however that linearity of the map on global sections will be immediate and independent on any further assumptions on the AQ degree. This is based on the following general result.

\begin{LEM}\label{rfggf784gfhf983hf89h393}
Let $\{e\}\ra \Ac\ra \Gc \ra \Cc\ra \{e\}$ be a short exact sequence of sheaves of groups on a topological space $X$ and suppose $\Ac$ and $\Cc$ are abelian. Then the induced map on cohomology 
\[
H^0\big(X, \Cc\big)\stackrel{\dt}{\lra} H^1\big(X, \Ac\big)
\]
will be linear.
\end{LEM}

\begin{proof}
The map $\dt$ is constructed as follows. Let $U = (U_i)\ra X$ be an open covering of $X$ and $c\in H^0(X, \Cc)$. Then $c$ defines local sections $c_i = c|_{U_i}$. Since $\{e\}\ra \Ac\ra \Gc \ra \Cc\ra \{e\}$ is a short exact sequence of sheaves we will have a cochain $g = (g_i)\in C^0(U, \Gc)$ where $g_i\in \Gc(U_i)$ and such that $g_i \mapsto c_i$ for all $i$. Observe that $(g_ig_j^{-1})_{ij}$ will define a 1-cocycle valued in $\Cc$. If $[(g_ig_j^{-1})_{ij}]$ denotes its cohomology class we can define,
\begin{align}
\dt : c \longmapsto \big[(g_ig_j^{-1})_{ij}\big].
\label{fnciubiugh8g4hg84}
\end{align}
This mapping depends on the cochain $g= (g_i)$ up to boundary terms and so \eqref{fnciubiugh8g4hg84} gives a well defined mapping $H^0(X, \Cc) \ra H^1(X, \Ac)$. Grothendieck in \cite{GROTHNONAB} observed that there exists a group action $H^0(X, \Cc)\times H^1(X, \Ac) \stackrel{\star}{\ra} H^1(X, \Ac)$,
\begin{align*}
(c, a) \longmapsto c\star a&&\mbox{given on cocycles by}&&
(c\star a)_{ij} = g_ia_{ij}g_j^{-1}.
\end{align*}
where $g_i\mapsto c_i = c|_{U_i}$. Evidently, by construction of the boundary map $\dt$ in \eqref{fnciubiugh8g4hg84} we find: 
\begin{align}
\dt(c) = c\star \{e\}, 
\label{dkldmnvldnbjkbbvyve}
\end{align}
for $\{e\}$ the base-point in $H^1(X, \Ac)$. With this formulation it is clear that
\begin{align}
\dt(c\cdot c^\p) = c\star \big(c^\p\star \{e\}\big)
\label{nfkdnkvjndkjvndkjndk}
\end{align}
where $c\cdot c^\p$ above is formed with respect to the natural group structure on $H^0(X, \Cc)$. In using that $\dt$ preserves the base-point we have $\dt(\{e\}) = \{e\}$. And hence from \eqref{dkldmnvldnbjkbbvyve} that $\dt(\{e\}) = \{e\} = \{e\}\star \{e\}$. With this identity we can therefore conclude from \eqref{nfkdnkvjndkjvndkjndk} that $\dt(cc^\p) = \dt(c)\star\dt(c^\p)$. Hence the subset $\dt(H^0(X, \Cc))\subset \mbox{\v H}^1(X, \Ac)$ admits a group structure with $\dt$ a homomorphism onto its image. If $\Ac$ is abelian then $\mbox{\v H}^1(X, \Ac) \cong H^1(X, \Ac)$ is an abelian group and therefore, with this isomorphism we see that $\dt : H^0(X, \Cc)\ra H^1(X, \Ac)$ will be a homomorphism of groups.  As a result when $\Cc$ and $\Ac$ are abelian, these groups are vector spaces and $\dt$ is a linear map.
\end{proof}

\noindent
The desired corollary of Lemma \ref{rfggf784gfhf983hf89h393}, mentioned in the remarks leading up to this lemma, is now  the following.

\begin{COR}\label{nkjbvkebubeufifif3ofno3}
Let $\Hc = (\Hc_j)_{j\geq0}$ be a sheaf of AQ central series on a space $X$ with linearisation $\Ac$. Then the differential $\pt : H^0(X, \Ac)\ra H^1(X, \Ac[1])$ in the primary complex of $\Hc$ is linear. It is given by
\[
\q^\p \stackrel{\pt^1_{j+1}}{\longmapsto} \q^\p\star \{e\}
\]
for $\q^\p\in H^0\big(X, \Ac_j\big)$.
\end{COR}

\begin{proof}
For each $j$ the differential $\pt$ is induced from the short exact sequence of sheaves
\[
\{e\} 
\lra
\Ac_{j+1} 
\lra 
\Hc_{j}/\Hc_{j+2}
\lra
\Ac_{j}
\lra 
\{e\}
\]
and by assumption $\Ac_{j+1}$ and $\Ac_{j}$ are abelian.
\end{proof}

\begin{REM}
\emph{We can also drop the condition that the AQ series in Corollary \ref{nkjbvkebubeufifif3ofno3} be central. More generally, for any AQ normal series $\Hc$ we can form its linearisation $\Ac$ and a mapping $\pt : H^0(X, \Ac)\ra H^1(X, \Ac[1])$ in Proposition \ref{rfbyuef8f9h39h3} (see \eqref{rf487fg873hfh89fh2j90j902j}). Lemma \ref{rfggf784gfhf983hf89h393} then asserts that this mapping will be linear.}
\end{REM}

\section{The Integration Problem}

\subsection{Torsors}
For a sheaf of groups $\Gc$ on a topological space $X$ one can, in general, form its cohomology in degrees zero and one. The degree zero cohomology of $\Gc$ comprise the global sections over $X$, $\Gc(X)$. This is a group. In contrast, the degree one cohomology of $\Gc$, ${\mbox{\v H}}^1\big(X, \Gc\big)$, is a pointed set and it is unclear how to endow it with any further structure, such as that of a group.\footnote{Of course, if $\Gc$ is a sheaf of \emph{abelian} groups, its cohomology exists in any degree and is, in addition, a finite dimensional vector space.} Elements of $\mbox{\v H}^1(X, \Gc)$ are represented by a covering $(\Ufr\ra X)$ together with $\Gc$-valued functions on each open set in the covering $\Ufr$ subject to relations on intersections and triple intersections. Such data is referred to as a \emph{torsor} or \emph{principle homogeneous $\Gc$-space} and, by construction, ${\mbox{\v H}}^1\big(X, \Gc\big)$ is the set classifying such objects (see e.g., \cite[p. 190]{BRY}). The base-point  in ${\mbox{\v H}}^1\big(X, \Gc\big)$ corresponds to the trivial torsor which is constructed by reference to the identity section over $X$.\footnote{Any morphism of sheaves of groups $\Gc^\p\ra \Gc$ induces a morphism on $1$-cohomology $\mbox{\v H}^1(X, \Gc^\p) \ra \mbox{\v H}^1(X, \Gc)$. Let $(e)$ be the trivial group. It is the initial object in the category of groups and so we have the inclusion $(e)\ra \Gc$ with $e$ mapping to the identity section in $\Gc$. As a set, $\{e\}= \mbox{\v H}^1(X, \{e\})$. The basepoint in $\mbox{\v H}^1(X, \Gc)$ is then the image of $\{e\}= \mbox{\v H}^1(X, \{e\})$ under the morphism $(e)\ra \Gc$.}

\subsection{AQ Series, Linearisations and Torsors} 
The purpose of a sheaf of AQ normal series $\Hc = (\Hc_j)_{j\geq0}$ is to study the $1$-cohomology of the sheaf of groups $\Hc_j$ by reference to the $1$-cohomology of the linearisation of $\Hc$, being $H^1\big(X,\Ac\big)$. The problem of starting with a element in the cohomology of the linearisation, $\q\in H^1\big(X,\Ac\big)$, and finding a $\Hc_j$-torsor `linearising' to $\q$ is what we will term the `integration problem'. To obtain a more precise formulation, let $X$ be a topological space and $\Hc = (\Hc_j)_{j\geq0}$ a sheaf of AQ normal series on $X$. With $\Ac = \oplus_j \Ac_j$ the linearisation of $\Hc$ we have a short exact sequence of groups,
\[
\{e\}
\lra 
\Hc_{j+1}
\lra
\Hc_j 
\lra 
\Ac_j
\lra
\{e\}
\]
leading to a long exact sequence of pointed sets, terminating on the following piece:
\begin{align}
\cdots 
\lra 
\mbox{\v H}^1\big(X, \Hc_{j+1}\big)
\lra 
\mbox{\v H}^1\big(X, \Hc_j\big)
\stackrel{\om_{j*}}{\lra}
H^1\big(X, \Ac_j\big).
\label{fcndkvbhjbvjrvjkenvke}
\end{align}
We term the map $\om_{j*}$ in the following.

\begin{DEF}\label{rfh74h94hfjf03j09f3j}
\emph{Let $\Hc = (\Hc_j)_{j\geq0}$ be a sheaf of AQ normal series on a topological space $X$. For each $j$, the mapping of pointed sets $\om_{j*}$ in \eqref{fcndkvbhjbvjrvjkenvke} is referred to as the \emph{$j$-th linearisation map}. It sends any $\Hc_j$-torsor $h$ to its \emph{linearisation} $\om_{j*}(h)\in H^1\big(X, \Ac_j\big)$.}
\end{DEF}

\noindent
Accordingly, Definition \ref{rfh74h94hfjf03j09f3j} above motivates the following.

\begin{DEF}\label{rnvrvkrnkjvnvjnkvnk}
\emph{Let $\Hc = (\Hc_j)_{j\geq0}$ be a sheaf of AQ normal series on a topological space $X$ with linearisation $\Ac$. The first cohomology $H^1(X, \Ac)$ is referred to as the \emph{space of $\Hc$-torsor linearisations}; and $H^1\big(X, \Ac_j\big)$ is the \emph{space of $\Hc_j$-torsor linearisations}.}
\end{DEF}

\begin{REM}\label{fjndkbvbvuvbuibvuibiv}
\emph{In our applications to supermanifolds, the objects in Definition \ref{rfh74h94hfjf03j09f3j} and Definition \ref{rnvrvkrnkjvnvjnkvnk} will undergo a change-in-terminology owing to their interpretations there. Indeed, what we are calling torsors here will be referred to as `supermanifold atlases'.}
\end{REM}

\noindent
With Definition \ref{rfh74h94hfjf03j09f3j} and \ref{rnvrvkrnkjvnvjnkvnk}, a natural question begs to be asked. It is this question which forms the basis of what we refer to as the \emph{integration problem}.

\begin{QUE}\label{rfuguybjbgjbgh4ih4u}
Let $\Hc = (\Hc_j)_{j\geq0}$ be a sheaf of AQ normal series on a topological space $X$. Then, given a homogeneous element $\q$ in the space of $\Hc$-torsor linearisations, when will $\q$ be the linearisation of some $\Hc_j$-torsor, for some $j$?
\end{QUE}

\noindent
Given $\q\in H^1\big(X, \Ac_j\big)$, if there exists some $\Hc_j$-torsor $h$ realising $\q$ as its linearisation, then we will say $\q$ `integrates' to $h$. Hence, to be more specific, Question \ref{rfuguybjbgjbgh4ih4u} will be referred to as the \emph{integration problem for $\q$}.

\begin{REM}
\emph{While $\q$ might integrate to $h$, this $h$ linearising to $\q$ certainly need not be unique. There will typically exist many other $h^\p$, different to $h$, and having $\q$ as its linearisation.}
\end{REM}

\subsection{The Integration Problem: Necessary Conditions} 
Sufficient conditions to resolving the integration problem are not generally apparent. A general necessary condition however is easier to deduce. As we will see below, it will involve the primary complex.

\begin{PROP}\label{fbiubiiidnono2o}
Let $\Hc$ be a sheaf of AQ central series on $X$. Denote by $\Ac$ the linearisation of $\Hc$. If the integration problem can be resolved for some $\q\in H^1\big(X, \Ac_j\big)$, then $\pt\q = 0$ where $\pt : H^1\big(X, \Ac_j\big)\ra H^2\big(X, \Ac_{j+1}\big)$ is the boundary map in the primary complex of $\Hc$.
\end{PROP}

\begin{proof}
For each $j$ we have a commutative diagram of sheaves,
\[
\xymatrix{
\ar@{=}[d] \Hc_j \ar[r] & \Hc_j/\Hc_{j+2}\ar[d]
\\
\Hc_j\ar[r] & \Ac_j.
}
\]
Hence on cohomology we get
\[
\xymatrix{
\ar@{=}[d] \mbox{\v H}^1\big(X, \Hc_j\big) \ar[r] & H^1\big(X, \Hc_j/\Hc_{j+1}\big)\ar[d]
\\
\mbox{\v H}^1\big(X, \Hc_j\big)\ar[r] & H^1\big(X, \Ac_j\big)\ar[d]^\pt
\\
& H^2\big(X, \Ac_{j+1}\big)\ar[d]
\\
&\vdots
}
\]
The proposition now follows from exactness of the column.
\end{proof}

\noindent
Hence for each $j$ we have the inclusion 
\begin{align}
\img\big\{\om_{j*} : \mbox{\v H}^1(X, \Hc_j) \ra H^1(X, \Ac_j)\big\}
\subseteq \ker\big\{\pt : H^1(X, \Ac_j)\ra  H^2(X, \Ac_{j+1})\big\}.
\label{rfu4h7h9fh3fj0j90}
\end{align}
In Proposition \ref{fbiubiiidnono2o} we obtained necessary conditions for resolving the integration problem. In particular, we can comfortably conclude the following: \emph{if $\q\in H^1(X, \Ac_j)$ and $\pt\q\neq0$, then there will not exist any $\Hc_j$-torsor with $\q$ as its linearisation}. Regarding sufficient conditions, this will involve the primary cohomology in degree one.

\begin{THM}\label{fjbvhjdbvhjbdjhd}
Let $\Hc = (\Hc_j)$ be a sheaf of AQ central series on $X$ with linearisation $\Ac$ and suppose its first primary cohomology vanishes, i.e., that ${\bf H}^1_\pt(X, \Ac) = (0)$. Then the following statements are equivalent for any $j$:
\begin{enumerate}[(i)]
	\item $\q\in H^1(X, \Ac_j)$ will integrate to a $\Hc_j$-torsor;
	\item  $\pt\q = 0$.
\end{enumerate}
\end{THM}

\begin{proof}
Theorem \ref{fbiubiiidnono2o} concerns the implication \emph{(i)} $\Rightarrow$ \emph{(ii)}. Conversely, suppose $\pt\q =0$. Recall from Theorem \ref{fbvuyrbvybriuvoievpie} the construction of the primary complex from the AQ central series $\Hc = (\Hc_j)$. From this construction, along with the identifications in \eqref{eh3hf983hf98h3} and the diagram in \eqref{fjnvfvjnvkjndkjvndk}, note that the following diagram will commute:
\begin{align}
\xymatrix{
& \mbox{\v H}^1\big(X, \Hc_j\big)\ar[dr]^{\om_{j*}}&
\\
H^0\big(X, \Ac_{j-1}\big) \ar[ur]^{\dt_j^1} \ar@{-->}[rr]_\pt & & H^1\big(X, \Ac_j\big)
}
\label{fnuibvurbiuenoinioe}
\end{align}
Now if  ${\bf H}^1_\pt(X, \Ac) = (0)$ then we can equate,
\[
\ker \big\{\pt: H^1(X, \Ac_j)\ra H^2(X, \Ac_{j+1}\big\}
=
\img\big\{\pt: H^0(X, \Ac_{j-1})\ra H^2(X, \Ac_{j})\big\}
\]
Hence if $\pt\q = 0$ we can write $\q = \om_{j*} \dt_j^1\q^\p$. In particular that $\q\in \img~\om_{j*}$ and we therefore obtain the implication $(ii)\Rightarrow (i)$, as desired.
\end{proof}

\section{Applications to Supermanifolds I: Generalities}

\subsection{Preliminaries}

\subsubsection{Supermanifolds}
A pair $(X, T^*_{X, -})$, with $X$ is a complex manifold and $T^*_{X, -}$ a holomorphic vector bundle, is referred to as a \emph{model}. It is the data required to define the notion of a supermanifold $\Xc$. We reserve the prefix `abstract' to refer to supermanifolds without a choice of covering. A \emph{supermanifold} is then an abstract supermanifold $\Xc$ equipped with a covering $\Uc = (\Uc_i)_{i\in I}$ and glueing data $\rho$ on intersections. It is denoted $((\Uc, \rho)\ra \Xc)$ or more simply $(\Uc\ra \Xc)$. Before describing supermanifolds abstractly we will firstly deliberate on the prototypical example, the \emph{split model}. To a model $(X, T^*_{X, -})$ the split model is the locally ringed space $S(X, T^*_{X, -}) = (X, \wedge^\bt T^*_{X, -})$. As described in more detail in \cite{BETTPHD,BETTOBSTHICK} any covering $U = (U_i)_{i\in I}$ for $X$ with prescribed glueing data will lift to give a covering and glueing data on the split model $\big(\widetilde U\ra S(X, T^*_{X, -})\big)$, giving rise then to a (split) supermanifold. An abstract supermanifold $\Xc$ modelled on $(X, T^*_{X, -})$ is a locally ringed space which is locally isomorphic to the split model $S(X, T^*_{X, -})$.  

\begin{DEF}\label{fbvrbvirbiunvneonvoie}
\emph{An abstract supermanifold is \emph{split} if it is isomorphic to the split model. Otherwise, it is non-split.}
\end{DEF}

\noindent
\begin{REM}\label{dkdjvjdkbvkbdvbebvkun}
\emph{If we take the model $(X, T^*_{X, -})$ to be comprised of a smooth resp. complex manifold and a smooth resp. complex-smooth vector bundle, then we can form the notion of smooth and complex-smooth supermanifolds analogously. With $T_{X, -}^*$ holomorphic we have complex analytic supermanifolds and it is these which we will be primarily concerned with in this article.}
\end{REM}

\subsubsection{Green's AQ Series}
Green in \cite{GREEN} laid the foundations for a classification of complex supermanifolds, building on Batchelor's classification in \cite{BAT} in the smooth setting. We review Green's work in the language used in this article. To a model $(X, T^*_{X, -})$ we have the sheaf of exterior algebras $\wedge^\bt T^*_{X, -}$. This is a finite, non-negatively graded, sheaf of $\Oc_X$-modules, where $\Oc_X$ is the structure sheaf of $X$ coming from its complex structure. Let $J$ be the irrelevant ideal and $J^k$ its $k$-th power. As $\Oc_X$-modules, $J \cong\oplus_{j>0}\wedge^jT^*_{X, -}$ and $J^k \cong \oplus_{j\geq k}\wedge^jT^*_{X, -}$. Now $\wedge^\bt T^*_{X, -}$ is a sheaf of supercommutative algebras with respect to the wedge product. Denote by $\mathcal Aut_{\Zbb_2}\wedge^\bt T^*_{X, -}$ those automorphisms which preserve the $\Zbb_2$-grading. Then with respect to $J$ we can form the following sheaf of groups,
\begin{align}
\Gc^{(k)}_{T^*_{X, -}} 
\stackrel{\Delta}{=}
\big\{\al\in \mathcal Aut_{\Zbb_2}\wedge^\bt T^*_{X,-}
\mid
\al(u) - u\in J^k
\big\}.
\label{fvnrnvnkjnkfjnwjkfnjkfnk}
\end{align}
Set $\mathrm G_{T^*_{X, -}} = \big(\Gc^{(k)}_{T^*_{X, -}}\big)_{k>1}$. The following properties, relevant for the purposes of this article, were derived by Green in \cite{GREEN}:
\begin{enumerate}[$\bt$]
	\item $\mathrm{G}_{T^*_{X, -}}$ is a finite collection of groups and, for all $k> \mathrm{rank}~T^*_{X, -}$, the $k$-th term, $\mathrm{G}_{T^*_{X, -}, k} = \Gc^{(k)}_{T^*_{X, -}}$ is trivial;
	\item for each $k$ we have an inclusion of sheaves of groups
	\[
	\mathrm{G}_{T^*_{X, -}, k+1} \subset \mathrm{G}_{T^*_{X, -}, k}
	\]
	realising $\mathrm{G}_{T^*_{X, -}, k+1}$ as a sheaf of normal subgroups of $\mathrm{G}_{T^*_{X, -}, k}$;
	\item the quotient $\mathrm{G}_{T^*_{X, -}, k}/\mathrm{G}_{T^*_{X, -}, k}$ is a sheaf of abelian groups;
	\item for each $k$, $\mathrm{G}_{T^*_{X, -}, k}$ is a sheaf of normal subgroups of $\mathrm{G}_{T^*_{X, -}, 2}$.
\end{enumerate}
With these bullet points we can deduce the following.

\begin{PROP}\label{rh74hf983j0j309d930k}
To any model $(X, T^*_{X, -})$ the collection $\mathrm{G}_{T^*_{X, -}}$ defines a sheaf of AQ normal series on $X$.\qed
\end{PROP}

\noindent
Onishchik in \cite{ONISHCLASS} observes that the AQ series $\mathrm{G}_{T^*_{X, -}}$ will  have degree $2$. With the aid of the Campbell-Baker-Hausdorf this observation of Onishichik is improved upon in \cite{BETTEMB}.

\begin{PROP}\label{cnjbvkjbrvkjneke}
The sheaf of AQ normal series $\mathrm{G}_{T^*_{X, -}}$ has AQ degree $4$.
\qed
\end{PROP}

\begin{DEF}
\emph{To a model $(X, T^*_{X, -})$, the AQ normal series $\mathrm{G}_{T^*_{X, -}}$ will be referred to as the \emph{AQ series of the model $(X, T^*_{X, -})$}.}
\end{DEF}

\noindent
In what follows we set $\Gc_{T^*_{X, -}} \stackrel{\Delta}{=} \mathrm G_{T^*_{X,-}, 2}$, for notational convenience.

\subsection{Classification of Supermanifolds}
Fix a model $(X, T^*_{X, -})$. Then any supermanifold $(\Uc\ra\Xc)$ modelled on $(X,T^*_{X, -})$ will define a class in the set $\mbox{\v H}^1\big(X, \Gc_{T^*_{X, -}}\big)$. This allows for a classification of supermanifolds. One of the main results in \cite{GREEN} is in presenting such a classification.

\begin{THM}\label{rfg78f93jf09j3903f3f3}
To a model $(X, T^*_{X, -})$ there exists an action of the group of global automorphisms $\Aut~T^*_{X, -}$ on $\mbox{\emph{\v H}}^1\big(X, \Gc_{T^*_{X, -}}\big)$ and, moreover, there exists a bijective correspondence of pointed sets:
\[
\frac{\mbox{\emph{\v H}}^1\big(X, \Gc_{T^*_{X, -}}\big)}{\Aut~T^*_{X, -}}
\cong
\left\{
\begin{array}{l}
\mbox{isomorphism classes of supermanifolds}\\
\mbox{modelled on $(X, T^*_{X, -})$}
\end{array}
\right\}
\]
with the $\Aut~T^*_{X, -}$-orbit of the base-point in $\mbox{\emph{\v H}}^1\big(X, \Gc_{T^*_{X, -}}\big)$ corresponding to the isomorphism class of the split model $\big(\Uc\ra S(X, T^*_{X, -})\big)$.\qed
\end{THM}

\noindent
Hence by Green's classification in Theorem \ref{rfg78f93jf09j3903f3f3}, a supermanifold modelled on $(X,T^*_{X, -})$ will be a representative of an element in $\mbox{\v H}^1\big(X, \Gc_{T^*_{X, -}}\big)$, up to an action of the global automorphisms $\Aut~T^*_{X, -}$.

\subsubsection{Strong Splitting}
The starting point behind the study of supermanifolds and higher obstructions in \cite{BETTHIGHOBS} is from Green's classification in Theorem \ref{rfg78f93jf09j3903f3f3}, which motivated the notion of a `strongly split supermanifold atlas'. We review this notion here, starting with the following.

\begin{DEF}
\emph{Let $(X, T^*_{X, -})$ be a model and $\mathrm G_{T^*_{X, -}}$ its AQ series. Any representative of an element in the $1$-cohomology $\mbox{\v H}^1\big(X, \mathrm G_{T^*_{X, -}, j}\big)$ is called 
\emph{a $j$-th order supermanifold atlas}, or simply `$j$-th order atlas'.}
\end{DEF}

\noindent
For any $j$ we have the inclusion $\mathrm G_{T^*_{X, -}, j}\ra \Gc_{T^*_{X, -}}$ inducing a map on cohomology $\mbox{\v H}^1\big(X, \mathrm G_{T^*_{X, -}, j}\big)\stackrel{\tau_{j*}}{\ra} \mbox{\v H}^1\big(X, \Gc_{T^*_{X, -}}\big)$. Hence any $j$-th order supermanifold atlas will define a supermanifold modelled on $(X, T^*_{X, -})$.

\begin{DEF}\label{tg478f79hf830j4444}
\emph{For each $j$ the image of a $j$-th order atlas under the map $\tau_{j*}:\mbox{\v H}^1\big(X, \mathrm G_{T^*_{X, -}, j}\big)\ra \mbox{\v H}^1\big(X, \Gc_{T^*_{X, -}}\big)$ is referred to as its \emph{associated supermanifold}. The atlas is \emph{split} if its associated supermanifold is split.}
\end{DEF}

\noindent
Conversely, we can refer to supermanifold atlases relative to a supermanifold.

\begin{DEF}
\emph{Let $(\Ufr\ra\Xfr)$ be a supermanifold modelled on $(X, T^*_{X, -})$. We say $(\Ufr\ra\Xfr)$ admits a \emph{level $j$ structure} if there exists a $j$-th order atlas whose image in $\mbox{\v H}^1\big(X, \Gc_{T^*_{X, -}}\big)$ coincides with the class defined by $(\Ufr\ra\Xfr)$. Similarly, starting from a $j$-th order supermanifold atlas, we say it admits a $j^\p$-th order structure, for some $j^\p> j$, if it lies in the image of an element in $\mbox{\v H}^1\big(X, \mathrm G_{T^*_{X, -}, j^\p}\big)$.}
\end{DEF}

\noindent
By Definition \ref{fbvrbvirbiunvneonvoie} a supermanifold is split if it is isomorphic to its split model. The classification in Theorem \ref{rfg78f93jf09j3903f3f3} then asserts: \emph{if a supermanifold modelled on $(X, T^*_{X, -})$ is split then it will lie in the orbit of the base-point under the action of $\Aut~T^*_{X, -}$}. Disregarding the $\Aut~T^*_{X, -}$-action then, we arrive at the following notion of splitting.

\begin{DEF}\label{ruieff98fh83f309jf3jf4rf}
\emph{A $j$-th order supermanifold atlas modelled on $(X, T^*_{X, -})$ is \emph{strongly split} if its associated supermanifold represents the base-point in $\mbox{\v H}^1\big(X, \Gc_{T^*_{X, -}}\big)$.}
\end{DEF}

\noindent
To justify the terminology in the above definition, any strongly split supermanifold atlas is split but not necessarily conversely.

\subsection{Obstructions to Existence and Splitting}

\subsubsection{Obstructions to Splitting}
With Proposition \ref{rh74hf983j0j309d930k} and \ref{cnjbvkjbrvkjneke} we can apply the derivations in previous sections to the AQ series $\mathrm G_{T^*_{X, -}}$. Let $\Ac_{T^*_{X, -}}$ denote the linearisation of $\mathrm G_{T^*_{X, -}}$. It is a $\Zbb$-graded sheaf of abelian groups and in other articles by the auther, e.g., in \cite{BETTOBSTHICK, BETTEMB, BETTHIGHOBS}, it is referred to as the \emph{obstruction sheaf of the model $(X, T^*_{X, -})$}. This is because its $1$-cohomology, referred to as the \emph{obstruction space of $(X, T^*_{X, -})$}, houses the obstruction classes to splitting. More precisely:

\begin{DEF}\label{rgf874f9h39830j39}
\emph{Fix a model $(X, T^*_{X, -})$. Then
\begin{enumerate}[$\bt$]
	\item the linearisation $\Ac_{T^*_{X, -}}$ of the AQ series $\mathrm G_{T^*_{X, -}}$ is referred to as the \emph{obstruction sheaf of $(X, T^*_{X, -})$};
	\item the $1$-cohomology of the obstruction sheaf of $(X, T^*_{X,-})$, $H^1\big(X, \Ac_{T^*_{X, -}}\big)$, is referred to as the \emph{obstruction space of $(X, T^*_{X, -})$} and;
	\item the linearisation of a given supermanifold atlas modelled on $(X, T^*_{X, -})$ is referred to as its \emph{obstruction to splitting};
\end{enumerate}}
\end{DEF}

\noindent
The term `obstruction to splitting' in Definition \ref{rgf874f9h39830j39} can be traced back to the works of Berezin, collected in \cite{BER}. Indeed, it is shown there that if the obstructions to splitting a supermanifold vanishes, then the supermanifold is split. This statement relies on the following statements which we present here without proof, for completeness of exposition: 
\begin{enumerate}[$\bt$]
	\item if the obstruction to splitting a $j$-th order supermanifold atlas vanishes, then the supemanifold atlas will admit a $(j+1)$-th order structure;
	\item with respect to a model $(X, T^*_{X, -})$ with $q = \mathrm{rank}~T^*_{X, -}$, any $j$-th order supermanifold atlas with $j> q$ is strongly split and;
	\item any $j$-th order supermanifold atlas which admits a $j^\p$-th order structure, for $j^\p> q$, is strongly split.
\end{enumerate}
The integration problem here is then concerned with the following question:

\begin{QUE}\label{fbvuirbvuviuneinvoee}
To a homogeneous class $\q$ in the obstruction space of a model, when will there exist a supermanifold atlas realising $\q$ as its obstruction to splitting?
\end{QUE}

\subsubsection{The Obstruction Complex of a Model}
In Theorem \ref{fvnrbvurnvnvkmvkel}
we constructed a complex of vector spaces from the data of a sheaf of AQ normal series of degree $4$. By Proposition \ref{rh74hf983j0j309d930k} and \ref{cnjbvkjbrvkjneke} we have precisely such data $\mbox G_{T^*_{X, -}}$ associated to any model $(X, T^*_{X, -})$. Since the linearisation of $\mbox G_{T^*_{X, -}}$ are the obstruction sheaves of the model $(X, T^*_{X, -})$ we arrive now at the notion of an `obstruction complex'.

\begin{DEF}
\emph{Fix a model $(X, T^*_{X, -})$ with AQ series $\mbox G_{T^*_{X, -}}$ and linearisation $\Ac_{T^*_{X, -}}$. Then the complex $\big(\Cc^\bt (X, \Ac_{T^*_{X, -}}), \pt\big)$ with $n$-th term
\begin{align*}
\Cc^n\big(X, \Ac_{T^*_{X, -}}\big) = H^n\big(X, \Ac_{T^*_{X,-}}\big)
&&
\mbox{and}
&&
\pt : 
H^n\big(X, \Ac_{T^*_{X,-}}\big)
\lra 
H^n\big(X, \Ac_{T^*_{X,-}}[1]\big),
\end{align*}
whose existence is guaranteed by Proposition $\ref{cnjbvkjbrvkjneke}$ and Theorem $\ref{fvnrbvurnvnvkmvkel}$, will be referred to as  the \emph{obstruction complex of the model $(X, T^*_{X, -})$.} Its cohomology will be referred to as the \emph{obstruction cohomology of $(X, T^*_{X, -})$}, denoted ${\bf H}_\pt^\bt \big(X, \Ac_{T^*_{X, -}}\big)$.}
\end{DEF}

\noindent
The significance of the higher order terms in the obstruction complex and cohomology are presently unclear. In degrees zero, one and two which constitutes the `primary complex' by Definition \ref{g78f7f93hf83h803}, we have a clear relation to the integration problem posed in Question \ref{fbvuirbvuviuneinvoee}.

\subsubsection{The Primary Complex}
To a model $(X, T^*_{X, -})$, its primary complex is the following complex of maps of $\Zbb$-graded vector spaces:
\begin{align}
\xymatrix{
H^0\big(X, \Ac_{T^*_{X, -}}[-1]\big)
\ar[r]^\pt
& 
H^1\big(X, \Ac_{T^*_{X, -}}\big)
\ar[r]^\pt 
& 
H^2\big(X, \Ac_{T^*_{X, -}}[1]\big)
}
\label{rbffh93hf9hf3fj390jf034}
\end{align}
The maps $\pt$ are linear and fit into the larger, obstruction complex. For the applications we give here however, knowledge about the primary complex is sufficient.  In \cite{BETTPHD, BETTOBSTHICK}, Question \ref{fbvuirbvuviuneinvoee} was raised and motivated much of the study presented there. A necessary condition to resolving Question \ref{fbvuirbvuviuneinvoee} was identified in \cite[Theorem 3.14, p. 31]{BETTOBSTHICK}. We recover this result below as an application of Proposition \ref{fbiubiiidnono2o}.

\begin{THM}\label{rfh89hf3jf903jf3}
To any model $(X, T^*_{X, -})$ the boundary map $\pt: H^1\big(X, \Ac_{T^*_{X, -}}\big)\ra H^2\big(X, \Ac_{T^*_{X, -}}[1]\big)$ in its primary complex in \eqref{rbffh93hf9hf3fj390jf034} measures the failure for a homogeneous element $\q\in H^1\big(X, \Ac_{T^*_{X, -}}\big)$ to represent an obstruction to splitting a supermanifold atlas modelled on $(X, T^*_{X, -})$.\qed
\end{THM}

\begin{REM}
\emph{We note that Theorem \ref{rfh89hf3jf903jf3} is not in itself a new result and was known at least to Eastwood and LeBrun in \cite{EASTBRU}. The derivation of Theorem \ref{rfh89hf3jf903jf3} by Eastwood and LeBrun is different than here and involves spectral sequences.}
\end{REM}

\noindent
In the course of resolving Question \ref{fbvuirbvuviuneinvoee}, three categories of classification were identified in \cite{BETTOBSTHICK} pertaining to homogeneous elements $\q$ in the obstruction space of a model. We say $\q$ represents:
\begin{enumerate}[(i)]
	\item \emph{a supermanifold} if there exists a supermanifold atlas realising $\q$ as its obstruction to splitting;
	\item \emph{a pseudo-supermanifold} if there does \emph{not} exist any such supermanifold atlas, and yet $\pt\q = 0$ and;
	\item \emph{an obstructed thickening} if $\pt\q\neq0$.
\end{enumerate}
In contrast with obstructed thickenings, where examples were constructed over the projective plane in \cite{BETTOBSTHICK}, examples of pseudo-supemanifolds are not so forth-coming. In what follows we will see how the obstruction cohomology of a model `anti-obstructs' the existence of pseudo-supermanifolds.

\begin{THM}\label{rfuihf3hf98h3f89h398fh3}
Let $(X, T^*_{X, -})$ be a model and suppose its first obstruction cohomology vanishes, i.e., that ${\bf H}_\pt^1(X, \Ac_{T^*_{X, -}}) = (0)$. Then any homogeneous element $\q\in H^1(X, \Ac_{T^*_{X, -}})$ will be the obstruction to splitting some supermanifold atlas if and only if $\pt\q = 0$.
\end{THM}

\begin{proof}
This is precisely an adaptation of Theorem \ref{fjbvhjdbvhjbdjhd} to the present setting.
\end{proof}

\noindent
Hence if ${\bf H}_\pt^1(X, \Ac_{T^*_{X, -}}) = (0)$, there will not exist any pseudo-supermanifolds modelled on $(X, T^*_{X, -})$. It is in this sense that ${\bf H}_\pt^1(X, \Ac_{T^*_{X, -}})$ `anti-obstructs' the existence of pseudo-supermanifolds. In what follows we will deduce further relations between the first obstruction cohomology of a model and `goodness' of a model in the sense of \cite{BETTHIGHOBS}.

\subsection{Exotic Atlases and Good Models}
One of the central concepts underpinning much of the work in \cite{BETTHIGHOBS} is the notion of an \emph{exotic atlas}, discussed by Donagi and Witten in \cite{DW1}, leading then to the definition of a `good model'. We present the definition from \cite{BETTHIGHOBS}. 

\begin{DEF}\label{rfuihfh39f8h93hf03}
\emph{A $j$-th order supermanifold atlas is said to be \emph{exotic} if:
\begin{enumerate}[$\bt$]
	\item it is strongly split and;
	\item defines a non-vanishing obstruction to splitting.
\end{enumerate}}
\end{DEF}

\begin{DEF}\label{fnkjnvkubvuinoifoi3fijio3jfiojo}
\emph{A model $(X, T^*_{X, -})$ is said to be \emph{good} if there do not exist any exotic supermanifold atlases modelled on $(X, T^*_{X, -})$.}
\end{DEF}

\noindent
One of the main results in \cite{BETTHIGHOBS} concerned necessary and sufficient conditions for a model to be `good'. Presently, we will see how the first obstruction cohomology of a model is related to `goodness' of the model. We begin with the following existence result.

\begin{PROP}\label{rfh489fh4hf4hf8fj39f}
Let $(X, T^*_{X, -})$ be a model and suppose ${\bf H}_\pt^1(X, \Ac_{T^*_{X, -}}) = (0)$. Then for any element $\q\in H^1(X, \Ac_{T^*_{X, -}, j})$ satisfying $\pt\q=0$, there will exist a $j$-th order supermanifold atlas which is strongly split and realises $\q$ as its obstruction to splitting. 
\end{PROP}

\begin{proof}
Assuming ${\bf H}_\pt^1\big(X, \Ac_{T^*_{X, -}}\big) = (0)$ we can equate,
\begin{align*}
\ker\big\{ \pt: H^1\big(X, \Ac_{T^*_{X, -}}\big)\ra~&H^2\big(X, \Ac_{T^*_{X, -}}[1]\big)\big\}
\\
&=
\img\big\{\pt: H^0\big(X, \Ac_{T^*_{X, -}}[-1]\big)\ra H^1\big(X, \Ac_{T^*_{X, -}}\big)\big\}.
\end{align*}
Hence if $\pt\q =0$ we can write $\q = \pt\q^\p$, for some $\q^\p\in H^0\big(X,  \Ac_{T^*_{X, -}, j-1}\big)$. We reproduce the commutative diagram in \eqref{fnuibvurbiuenoinioe} below:
\begin{align}
\xymatrix{
& \mbox{\v H}^1\big(X, \mathrm G_{T^*_{X, -}, j}\big)\ar[dr] &
\\
\ar[ur]^\dt H^0\big(X, \Ac_{T^*_{X, -}, j-1}\big) \ar[rr]_\pt & & H^1\big(X, \Ac_{T^*_{X, -},j}\big)
}
\label{Prfg84gf848f7h4fh9}
\end{align}
Hence with $\pt\q= 0$ giving $\q =\pt \q^\p$, commutativity of \eqref{Prfg84gf848f7h4fh9} shows that $\q$ will be the obstruction to splitting the atlas $\dt\q^\p$. Now note that $\dt$ fits into the following exact sequence of pointed sets,
\[
\ldots
\lra
H^0\big(X, \Ac_{T^*_{X, -}, j-1}\big)
\stackrel{\dt}{\lra} 
\mbox{\v H}^1\big(X, \mathrm G_{T^*_{X, -}, j}\big)
\stackrel{\iota_*}{\lra} 
\mbox{\v H}^1\big(X, \mathrm G_{T^*_{X, -}, j-1}\big)
\lra\cdots
\]
Therefore $\iota_*\dt\q^\p = \{e\}$. Now for any $j$ the following diagram will commute,
\begin{align}
\xymatrix{
\ar[dr]_{\tau_{j*}} \mbox{\v H}^1\big(X, \mathrm G_{T^*_{X, -}, j}\big)\ar[rr]^{\iota_*} & & \mbox{\v H}^1\big(X, \mathrm G_{T^*_{X, -}, j-1}\big)\ar[dl]^{\tau_{j-1*}}
\\
& \mbox{\v H}^1\big(X, \Gc_{T^*_{X, -}}\big) &
}
\label{brtgtyfhf7hfhfoij3}
\end{align}
Since $\iota_*\dt\q^\p = \{e\}$ commutativity of \eqref{brtgtyfhf7hfhfoij3} implies $\tau_{j*}\dt\q^\p = \{e\}$, further implying $\dt\q^\p$ is strongly split.
\end{proof}

\noindent
A corollary now is the following relation between the obstruction space and obstruction cohomology, resulting from `goodness' of a model.

\begin{THM}\label{rfg7gfi3hfiuwhfioh}
Let $(X, T^*_{X, -})$ be a good model. Then ${\bf H}_\pt^1(X, \Ac_{T^*_{X, -}}) = (0)$ if and only if $\pt: H^1\big(X, \Ac_{T^*_{X, -}}\big) \ra H^2\big(X, \Ac_{T^*_{X, -}}[1]\big)$ is injective.
\end{THM}

\begin{proof}
If $\pt: H^1\big(X, \Ac_{T^*_{X, -}}\big) \ra H^2\big(X, \Ac_{T^*_{X, -}}[1]\big)$ is injective, then ${\bf H}_\pt^1(X, \Ac_{T^*_{X, -}}) = (0)$. Conversely, with $(X, T^*_{X, -})$ a good model, suppose ${\bf H}_\pt^1(X, \Ac_{T^*_{X, -}}) = (0)$. Then by Proposition \ref{rfh489fh4hf4hf8fj39f} we know that for any homogeneous element $\q$ in the obstruction space with $\pt\q = 0$, there will exist a strongly split supermanifold atlas realising $\q$ as its obstruction to splitting. If $\q\neq0$, this supermanifold atlas will be exotic by Definition \ref{rfuihfh39f8h93hf03}, contradicting our assumption that $(X, T^*_{X, -})$ is a good model. Hence if $\pt\q = 0$, it must be the case that $\q = 0$. Since $\pt$ is a linear map between vector spaces, this condition means $\pt$ is injective.
\end{proof}

\noindent
On a Riemann surface then we have the corollary.

\begin{COR}\label{fjcdkbvjdbvjhbvrbvbi}
Let $(X, T^*_{X, -})$ be a good model with $X$ a Riemann surface. Then ${\bf H}_\pt^1(X, \Ac_{T^*_{X, -}}) = (0)$ if and only if $H^1\big(X, \Ac_{T^*_{X, -}}\big) = (0)$.
\end{COR}

\begin{proof}
For dimensional reasons $H^2(X, \Fc) = (0)$ for \emph{any} abelian sheaf $\Fc$ on a Riemann surface $X$. 
\end{proof}

\noindent
Hence, on a Riemann surface, if the obstruction space of a good model does not vanish, then neither can its first obstruction cohomology.

\section{Applications to Supermanifolds II: Lift, Project, Split}
\label{fcbeuyvyegv8g79g93h93}

\subsection{Berezin's Lifting Theorem and Uniqueness}

\subsubsection{Existence}
We recall here the notions in Definition \ref{rgf874f9h39830j39} and the subsequent discussion there. Fix a model $(X, T^*_{X,-})$. Any supermanifold atlas $(\Uc\ra\Xc)$ modelled on $(X, T^*_{X, -})$ will define an obstruction to splitting. If $(\Uc\ra\Xc)$ has order $j$, it will define an obstruction to splitting in the $j$-th graded component of the obstruction space $H^1\big(X, \Ac_{T^*_{X, -}}\big)$. If this obstruction vanishes, then $(\Uc\ra\Xc)$ will admit a $(j+1)$-th order structure. This last statement is based on the observation that the following sequence of pointed sets is exact:\footnote{recall that it is a piece of the long exact sequence on cohomology induced from the short exact sequence of sheaves $\{e\}\ra \mathrm G_{T^*_{X, -}, j+1}\ra\mathrm G_{T^*_{X, -}, j}\ra \Ac_{T^*_{X, -}, j}\ra\{e\}$.}
\[
\mbox{\v H}^1\big(X, \mathrm G_{T^*_{X, -}, j+1}\big)
\lra 
\mbox{\v H}^1\big(X, \mathrm G_{T^*_{X, -}, j}\big)
\stackrel{\om_{j*}}{\lra}
H^1\big(X, \Ac_{T^*_{X, -}, j}\big).
\]
Berezin observed (see \cite[pp. 164-5]{BER}) that when $j$ is \emph{even}, any $j$-th order supermanifold atlas with vanishing obstruction will admit a \emph{$(j+2)$-th order structure}. In terms of cohomology, Berezin's observation can be formulated as follows. Firstly note that for any $j$ we have the following diagram of sheaves:
\begin{align}
\xymatrix{
& & \Ac_{T^*_{X,-}, j+1}\ar[d]
\\
\mathrm G_{T^*_{X, -}, j+2} \ar[r] &\ar@{=}[d] \mathrm G_{T^*_{X, -}, j} \ar[r] & \mathrm G_{T^*_{X, -}, j} /\mathrm G_{T^*_{X, -}, j+2} \ar[d]
\\
& \mathrm G_{T^*_{X, -}, j} \ar[r] & \Ac_{T^*_{X, -}, j}
}
\label{rfh37f973h8309j955g5}
\end{align}
where the middle row and the right-most column are short exact sequences. From \eqref{rfh37f973h8309j955g5} we get the following diagram on $1$-cohomology,
\begin{align}
\xymatrix{
& & H^1\big(X, \Ac_{T^*_{X,-}, j+1}\big)\ar[d]^{\iota_{j*}}
\\
 \mbox{\v H}^1\big(X, \mathrm G_{T^*_{X, -}, j+2}\big) \ar[r]^{\tau_{~j*}^{j+2}} & \mbox{\v H}^1\big(X, \mathrm G_{T^*_{X, -}, j}\big)\ar@{=}[d] \ar[r] & H^1\big(X, \mathrm G_{T^*_{X, -}, j}/ \mathrm G_{T^*_{X, -}, j+2}\big)\ar[d]^{p_j*} 
\\
 & \mbox{\v H}^1\big(X, \mathrm G_{T^*_{X, -}, j}\big) \ar[r]^{\om_{j*}} & H^1\big(X, \Ac_{T^*_{X, -}, j}\big)
}
\label{dkldnckjbdkeybe}
\end{align}
When $j=2\ell$ is even, Berezin's observation is:

\begin{THM}\label{hyegh9fh389f89j03j93k}
To any $x\in \mbox{\emph{\v H}}^1\big(X, \mathrm G_{T^*_{X, -}, 2\ell}\big)$ with $\om_{2\ell*}(x) = 0$, there exists $x^\p\in \mbox{\emph{\v H}}^1\big(X, \mathrm G_{T^*_{X, -}, 2\ell+2}\big)$ such that $\tau_{~2\ell*}^{2\ell+2}: x^\p\mapsto x$.\qed
\end{THM}

\subsubsection{Uniqueness}
To $x\in \mbox{\v H}^1\big(X, \mathrm G_{T^*_{X, -}, 2\ell}\big)$ if we think of $x$ as a thickening of order $2\ell$, to use the language in \cite{BETTOBSTHICK}, then Berezin in \cite[pp. 163-4]{BER} argues that $x$ will admit a unique, first order extension. Note that this is regardless of whether $\om_{2\ell*}(x)$ vanishes or not. When $\om_{2\ell*}(x) = 0$, Proposition \ref{hyegh9fh389f89j03j93k} guarantees the existence of a lift $\tau_{~2\ell*}^{2\ell+2}: x^\p\mapsto x$. It is the uniqueness now of this lift $x^\p$ which will concern us here. 
To present the result firstly recall that the right-most column in \eqref{rfh37f973h8309j955g5} is a short exact sequence and leads, on cohomology, to the following left exact piece:
\begin{align}
\{e\}
\ra 
H^0\big(X, \Ac_{T^*_{X, -}, j+1}\big)
\lra
H^0\big(X, \mathrm G_{T^*_{X, -}, j}&/ \mathrm G_{T^*_{X, -}, j+2}\big)
\label{rniurui4f3hf083f09j30ereree}
\\
\notag&
\stackrel{H^0(p_{j})}{\lra}
H^0\big(X, \Ac_{T^*_{X, -}, j}\big).
\end{align}
Hence $H^0\big(X, \Ac_{T^*_{X, -}, j+1}\big)$ is a subgroup of $H^0\big(X, \mathrm G_{T^*_{X, -}, j}/ \mathrm G_{T^*_{X, -}, j+2}\big)$ for each $j$. By Proposition \ref{cnjbvkjbrvkjneke} we know that $\mathrm G_{T^*_{X, -}, j}/ \mathrm G_{T^*_{X, -}, j+2}$ is abelian. Hence we can form the quotient on cohomology, giving
\[
\img~H^0(p_j) \cong H^0\big(X, \mathrm G_{T^*_{X, -}, j}/ \mathrm G_{T^*_{X, -}, j+2}\big)/ H^0\big(X, \Ac_{T^*_{X, -}, j+1}\big).
\]
In the case where $j = 2\ell$, we have:

\begin{THM}
\label{rfh48gh4hg4040gj490j}
Let $x\in \mbox{\emph{\v H}}^1\big(X, \mathrm G_{T^*_{X, -}, 2\ell}\big)$ and suppose $\om_{2\ell*}(x) = 0$. Then there exists a unique (i.e., one and only one) lift $\tau_{~2\ell*}^{2\ell+2}: x^\p\mapsto x$ if and only if $\img~H^0(p_{2\ell})$ is trivial.
\end{THM}

\noindent
Our proof of Theorem \ref{rfh48gh4hg4040gj490j} will involve another lifting phenomenon which was observed by Donagi and Witten in \cite{DW1}. We present this result in the section to follow and so our proof of Theorem \ref{rfh48gh4hg4040gj490j} will appear there.

\subsection{Vanish, Lift and Vanish}

\subsubsection{Donagi and Witten's Vanishing}
We will firstly set up some notation. Fix a model $(X, T^*_{X, -})$ with obstruction sheaf $\Ac_{T^*_{X, -}}$. Recall that it is $\Zbb$-graded. Hence it will be $\Zbb_2$-graded. Let $\Ac_{T^*_{X, -}, +}$ and $\Ac_{T^*_{X, -}, -}$ denote its even and odd graded components respectively. Then as $\Oc_X$-modules,
\begin{align*}
 \Ac_{T^*_{X, -}, +}
 =
 \bigoplus_{j>1}  \Ac_{T^*_{X, -}, 2j}
 &&
 \mbox{and}
 &&
 \Ac_{T^*_{X, -}; -}
 =
 \bigoplus_{j>0}  \Ac_{T^*_{X, -}, 2j+1}.
\end{align*}
The boundary maps in the obstruction complex of $(X, T^*_{X, -})$ increase the $\Zbb$-degree of the obstruction sheaf by one, i.e., 
\[
\pt^{n+1}_{j} : H^n\big(X, \Ac_{T^*_{X,-}, j}\big)
\lra 
 H^{n+1}\big(X, \Ac_{T^*_{X,-}, j+1}\big).
\]
Therefore $\pt$ interpolates between the even and odd graded components of the obstruction sheaf. This means, for each $n$, we have:
\begin{align}
\pt^{n+1}_{\pm} : 
H^n\big(X, \Ac_{T^*_{X,-}, \pm}\big)
\lra 
 H^{n+1}\big(X, \Ac_{T^*_{X,-}, \mp}\big)
\label{cnuiciuhehf893hf8hf3}
\end{align}
We will refer to $\pt^{n+1}_+$ resp. $\pt^{n+1}_-$ as the \emph{even} resp. \emph{odd} components of $\pt^{n+1}$. Donagi and Witten in \cite{DW1} observed:

\begin{PROP}\label{rfh89hf983hf893f93}
To any model $(X, T^*_{X, -})$ the odd component $\pt^1_-$ of $\pt^1$ in \eqref{cnuiciuhehf893hf8hf3} vanishes identically.\end{PROP}

\subsubsection{Lifting and Vanishing}
As a consequence of Proposition \ref{rfh89hf983hf893f93} we will obtain a lift of a certain map which we describe presently. Consider the following diagram of solid arrows:
\begin{align}
\xymatrix{
&\mbox{\v H}^1\big(X, \mathrm G_{T^*_{X, -}, 2\ell+3}\big)\ar[d]
\\
H^0\big(X, \Ac_{T^*_{X, -}, 2\ell+1}\big)\ar@{-->}[ur] \ar[r]^{\dt_{2\ell+2}} \ar[dr]_{\pt^1_-} & \mbox{\v H}^1\big(X, \mathrm G_{T^*_{X, -}, 2\ell+2}\big)\ar[d]^{\om_{2\ell+2*}}
\\
& H^1\big( X, \Ac_{T^*_{X, -}, 2\ell+2}\big)
}
\label{rnvebvuibvnekemlme}
\end{align}
Proposition \ref{rfh89hf983hf893f93} says $\pt^1_-= 0$ above. Hence by exactness of the column, we can deduce the existence of the following map, represented in \eqref{rnvebvuibvnekemlme} by the dashed arrow:
\begin{align}
\xymatrix{
\widetilde \dt_{2\ell+2} : H^0\big(X, \Ac_{T^*_{X, -}, 2\ell+1}\big) \ar@{-->}[rr] & & \mbox{\v H}^1\big(X, \mathrm G_{T^*_{X, -}, 2\ell+3}\big)}
\label{fcnkdbvjkbjvkjvnjekvle}
\end{align}
for each $\ell$. Hence as a result of Donagi and Witten's vanishing in Proposition \ref{rfh89hf983hf893f93} we can lift the boundary map $\dt_{2\ell+2}$ in \eqref{rnvebvuibvnekemlme} to $\widetilde\dt_{2\ell+2}$ in \eqref{fcnkdbvjkbjvkjvnjekvle}. This leads to a subsequent vanishing result. 

\begin{PROP}\label{rgf78f79hf93h98fh39h93}
For any $\ell$ the linearisation of the lift $\widetilde\dt_{2\ell+2}$ of $\dt_{2\ell+2}$ vanishes, i.e., that the composition,
\begin{align}
\widetilde\pt:  H^0\big(X, \Ac_{T^*_{X, -}, 2\ell+1}\big)
\stackrel{\widetilde \dt_{2\ell+2}}{\lra}
\mbox{\emph{\v H}}^1\big(X, \mathrm G_{T^*_{X, -}, 2\ell+3}\big)
\stackrel{\om_{2\ell+3*}}{\lra}
H^1\big(X, \Ac_{T^*_{X, -}, 2\ell+3}\big)
\label{rfh794fh97hf983hf93h}
\end{align}
vanishes.
\end{PROP}

\noindent
The argument in Proposition \ref{rgf78f79hf93h98fh39h93} on the vanishing of the linearisation of $\widetilde \dt_{2\ell+2}$ generalises and allows us to deduce a more powerful vanishing result.

\begin{THM}\label{dkcdnlnvdvbebvukenvl}
For any $\ell$ the boundary map $\dt_{2\ell+2}$ in \eqref{rnvebvuibvnekemlme} is constant and sends every global section in $H^0\big(X, \Ac_{T^*_{X, -}, 2\ell+1}\big)$ to the base-point in $\mbox{\emph{\v H}}^1\big(X, \mathrm G_{T^*_{X, -}, 2\ell+2}\big)$.
\end{THM}

\noindent
Donagi and Witten's vanishing in Proposition \ref{rfh89hf983hf893f93}, followed by the construction of the lift $\widetilde\dt$ in \eqref{fcnkdbvjkbjvkjvnjekvle}, the vanishing of it's linearisation in Proposition \ref{rgf78f79hf93h98fh39h93} leading finally to the vanishing in Theorem \ref{dkcdnlnvdvbebvukenvl} above will be referred to collectively as the \emph{Vanish-Lift-Vanish} principle. We present a proof of this principle in Appendix \ref{fklnvkbvjrbvnfkemlkmel}.

\subsubsection{Proof of Theorem $\ref{rfh48gh4hg4040gj490j}$}
The Vanish-Lift-Vanish principle, culminating in Theorem \ref{dkcdnlnvdvbebvukenvl}, will be a crucial ingredient in our proof of Theorem \ref{rfh48gh4hg4040gj490j}. Before embarking on this proof it will be useful to digress and discuss the following general observation by Grothendieck in \cite{GROTHNONAB}, reviewed by Brylinski in \cite[p. 160]{BRY} and mentioned in the proof of Lemma \ref{rfggf784gfhf983hf89h393}. 

\begin{LEM}\label{rfh794hf98hf80j309j30}
Let $\{e\}\ra\Ac\ra\Gc\ra\Cc\ra\{e\}$ be a short exact sequence of sheaves of groups on a space $X$. Then there exists an action $\star$ of $H^0(X, \Cc)$ on $\mbox{\emph{\v H}}^1(X, \Ac)$ such that: for any $a, a^\p$ in $\mbox{\emph{\v H}}^1(X, \Ac)$, their image coincides in $\mbox{\emph{\v H}}^1(X, \Gc)$ if and only if there exists some global section $c\in H^0(X, \Cc)$ such that $a^\p = c\star a$.\qed
\end{LEM}

\noindent
To the short exact sequence of sheaves $\{e\}\ra\Ac\ra\Gc\ra\Cc\ra\{e\}$ on $X$ let $\dt : H^0(X, \Cc) \ra \mbox{\v H}^1(X, \Ac)$ be the boundary map. In the proof of Lemma \ref{rfggf784gfhf983hf89h393} this boundary map $\dt$
 was related to the action $\star$ from Lemma \ref{rfh794hf98hf80j309j30}. Importantly, for our purposes, we have:

\begin{LEM}\label{ciuevuyegi3g89h30j9f}
Suppose $\dt : H^0(X, \Cc) \ra \mbox{\emph{\v H}}^1(X, \Ac)$ is trivial. Then for any $c\in H^0(X, \Cc)$ we have 
\[
c\star a = a.
\]
for all $a\in \mbox{\emph{\v H}}^1(X, \Ac)$.
\end{LEM}

\begin{proof}
From Lemma \ref{rfggf784gfhf983hf89h393} the boundary map $\dt$ is given by 
\[
\dt : c \longmapsto c\star e
\]
where $e\in H^1(X, \Ac)$ is the base-point. If $\dt$ is trivial, then $\dt(c) = e$ for all $c$. Using that $e\star e = e$ we find $\dt(c) = c\star e = e$ for all $c$. Now the action $\star$ is associative. As such, for any $c\in H^0(X, \Cc)$ and $a\in H^1(X, \Ac)$ we have:
\[
c\star a
=
c\star \big(e\star a\big)
=
\big(c\star e\big)\star a
=
\dt(c)\star a
=
e\star a
=
a.
\]
The lemma now follows.
\end{proof}

\noindent
We now resume our proof of Theorem \ref{rfh48gh4hg4040gj490j}. Firstly recall the diagram in \eqref{rfh37f973h8309j955g5}. Note that we can `fill it in' to get:
\begin{align}
\xymatrix{
& & \Ac_{T^*_{X,-}, j+1}\ar[d]
\\
\mathrm G_{T^*_{X, -}, j+2} \ar[d] \ar[r] &\ar@{=}[d] \mathrm G_{T^*_{X, -}, j} \ar[r] & \mathrm G_{T^*_{X, -}, j} /\mathrm G_{T^*_{X, -}, j+2} \ar[d]
\\
\mathrm G_{T^*_{X, -}, j+1} \ar[d] \ar[r] & \mathrm G_{T^*_{X, -}, j} \ar[r] & \Ac_{T^*_{X, -}, j}
\\
\Ac_{T^*_{X, -}, j+1}& & 
}
\label{rfh37f973h862626262309j955g5sss}
\end{align}
And hence on cohomology,
\begin{align}
\xymatrix{
& H^0\big(X, \Ac_{T^*_{X, -}, j+1}\big)\ar[d] & 
\\
H^0\big(X, \mathrm G_{T^*_{X, -}, j}/ \mathrm G_{T^*_{X, -}, j+2}\big)\ar[r] & \mbox{\v H}^1\big(X, \mathrm G_{T^*_{X, -}, j+2}\big)\ar[d]_{\tau^{j+2}_{~j+1*}} \ar[r]^{\tau^{j+2}_{~j*}} & \mbox{\v H}^1\big(X, \mathrm G_{T^*_{X, -}, j}\big)\ar@{=}[d]
\\
& \mbox{\v H}^1\big(X, \mathrm G_{T^*_{X, -}, j+1}\big) \ar[r]& \mbox{\v H}^1\big(X, \mathrm G_{T^*_{X, -}, j}\big)
}
\label{fkncdnvvrfifio4jio4ftnkt}
\end{align}
By Lemma \ref{rfh794hf98hf80j309j30} we know that $\mbox{\v H}^1\big(X, \mathrm G_{T^*_{X, -}, j+2}\big)$ carries actions by $H^0\big(X, \Ac_{T^*_{X, -}, j+1}\big)$ and $H^0\big(X, \mathrm G_{T^*_{X, -}, j}/ \mathrm G_{T^*_{X, -}, j+2}\big)$ with respect to which the maps $\tau^{j+2}_{~j+1*}$ and $\tau^{j+2}_{~j*}$ are invariant, respectively. Specialising to the case where $j = 2\ell$, Theorem \ref{dkcdnlnvdvbebvukenvl} and Lemma \ref{ciuevuyegi3g89h30j9f} guarantee that the action of $H^0\big(X, \Ac_{T^*_{X, -}, j+1}\big)$ will be trivial. If we additionally assume $\img~H^0(p_{2\ell*})$ is trivial, then by \eqref{rniurui4f3hf083f09j30ereree} the groups $H^0\big(X, \Ac_{T^*_{X, -}, j+1}\big)$ and $H^0\big(X, \mathrm G_{T^*_{X, -}, j}/ \mathrm G_{T^*_{X, -}, j+2}\big)$ will be isomorphic and hence that the action of the group $H^0\big(X, \mathrm G_{T^*_{X, -}, j}/ \mathrm G_{T^*_{X, -}, j+2}\big)$ must also trivial. Uniqueness of the lift $\tau_{~2\ell*}^{2\ell+2}: x^\p\mapsto x$ then follows from Lemma \ref{rfh794hf98hf80j309j30}. 
\qed

\subsection{Projectable Models}

\subsubsection{Existence}
Any model $(X, T^*_{X, -})$ gives rise to a class of supermanifolds and, rather than studying particular supermanifolds, we might prefer to study this class of supermanifolds instead. This was the subject of \cite{BETTHIGHOBS} where a characterisation of models as being `good' or otherwise was presented (see Definition \ref{fnkjnvkubvuinoifoi3fijio3jfiojo}). Presently we will consider another kind of characterisation. Recall that a supermanifold $\Xfr$ is a locally ringed space $(X, \Oc_\Xfr)$ where $\Oc_\Xfr$ is globally $\Zbb_2$-graded sheaf which is locally isomorphic to a sheaf of exterior algebras.  With $\Oc_X$ the sheaf capturing the complex structure of $X$, the global $\Zbb_2$-grading on $\Oc_\Xfr$ defines an ideal $\Jc$ such that $\Oc_\Xfr/\Jc = \Oc_X$. Hence $\Jc$ is the ideal sheaf of an embedding of spaces $i:X\subset \Xfr$. As discussed in \cite{DW1}, a highly relevant structure on a supermanifold for the purposes of theoretical physics is on the existence of a projection map $\pi: \Xfr\ra X$ with $\pi i = {\bf 1}_X$. Such a map allows for the reduction of measures defined on $\Xfr$ to measures on $X$ and hence allows for a workable notion of `integration on supermanifolds'.

\begin{DEF}
\emph{A supermanifold $\Xfr = (X, \Oc_\Xfr)$ is said to be \emph{projectable} if there exists a projection map $\pi : \Xfr\ra X$.}
\end{DEF}

\noindent
Donagi and Witten in \cite{DW1} identify a collection of classes which obstruct the existence of a projection map in analogy with the obstructions to splitting discussed in the present article. For our purposes we will only need to know the following (see \cite[p. 18]{DW1}):

\begin{LEM}\label{cndnvjdbjmdbjfbku3f}
Any obstruction to the existence of a projection map coincides with an obstruction to splitting a supermanifold atlas of even order.\qed
\end{LEM}

\noindent
We now consider the following feature definable for models and thereby refining their classification.

\begin{DEF}\label{reiugfgf783h983hf89h30}
\emph{A model $(X, T^*_{X, -})$ is said to be \emph{projectable} if \emph{any} supermanifold modelled on $(X, T^*_{X, -})$ is projectable.}
\end{DEF}

\noindent
In Theorem \ref{rfuihf3hf98h3f89h398fh3} we deduced an interesting, geometric consequence of the vanishing of the first obstruction cohomology of a model. Subsequently we related this assumption to the notion of `goodness' of a model (see Proposition \ref{rfh489fh4hf4hf8fj39f} and Theorem \ref{rfg7gfi3hfiuwhfioh}). Presently, we find a relation to projectability.

\begin{THM}\label{fvndvkbvubvknenvkno}
Let $(X, T^*_{X, -})$ be a model with vanishing first obstruction cohomology, i.e., with ${\bf H}^1_\pt \big(X, \Ac_{T^*_{X, -}}\big) = (0)$. Then $(X, T^*_{X, -})$ is projectable.
\end{THM}

\begin{proof}
By Lemma \ref{cndnvjdbjmdbjfbku3f} we need only confirm, under the assumption ${\bf H}^1_\pt \big(X, \Ac_{T^*_{X, -}}\big) = (0)$, that the obstruction to splitting any supermanifold atlas of even order will vanish. 
Central to our proof is Theorem \ref{dkcdnlnvdvbebvukenvl}. Recall from Theorem \ref{rfh89hf3jf903jf3} that any supermanifold atlas modelled on $(X, T^*_{X, -})$ will define an obstruction to splitting in a homogeneous component in the kernel $\ker\big\{\pt : H^1\big(X, \Ac_{T^*_{X, -}}\big)\ra H^1\big(X, \Ac_{T^*_{X, -}}[1]\big)\big\}$. Since ${\bf H}^1_\pt\big(X, \Ac_{T^*_{X, -}}\big) = (0)$ we can equate,
\begin{align}
\ker\big\{\pt : H^1\big(X, \Ac_{T^*_{X, -}}\big)&\ra H^1\big(X, \Ac_{T^*_{X, -}}[1]\big)\big\} 
\notag\\
&=
\img\big\{\pt : H^1\big(X, \Ac_{T^*_{X, -}}[-1]\big)\ra H^1\big(X, \Ac_{T^*_{X, -}}\big)\big\}.
\label{rnii4fi7fh8fh83hf0333330}
\end{align}
Hence any obstruction will come from a homogeneous, global section $u$ of $\Ac_{T^*_{X, -}}$.
If $u$ is odd, i.e., lies in $H^0\big(X, \Ac_{T^*_{X, -}; -}\big)$ then by Theorem \ref{dkcdnlnvdvbebvukenvl} we know $\pt(u) = 0$. Hence if a supermanifold atlas modelled on $(X, T^*_{X,-})$ defines a non-vanishing obstruction to splitting, it must come from a global section of $\Ac_{T^*_{X, -}; +}$ and, in particular, lie in $H^1\big(X, \Ac_{T^*_{X, -};-}\big)$, i.e., the supermanifold atlas must have odd order. Hence, the obstruction to splitting any supermanifold atlas of even order must vanish. 
\end{proof}

\noindent
The contrapositive statement to Theorem \ref{fvndvkbvubvknenvkno} is then the following corollary on the non-vanishing of the obstruction cohomology.

\begin{COR}\label{fg39fh983hf8380f30j}
Let $(X, T^*_{X, -})$ be a model and suppose there exists a non-projectable supermanifold modelled on $(X, T^*_{X, -})$. Then the first obstruction cohomology of the model $(X, T^*_{X, -})$ will not vanish. \qed
\end{COR}

\begin{EX}\label{dkvndbvjkbukvbuevevnkek}
The subject of the works by Donagi and Witten in \cite{DW1, DW2} is to show that moduli space of super Riemann surfaces $\Mfr_g$ is non-projectable in genus $g\geq 5$.  Now $\Mfr_g$ is a superspace modelled on the moduli space of spin curves $\Scl\Mcl_g$ and vector bundle $T^*_{\Scl\Mcl_g, -}$ whose fiber over a spin curve $(C, T_C^{1/2})$ is the $1$-cohomology $T^*_{\Scl\Mcl_g, -}|_{(C, T_C^{1/2})} = H^1(C, T_C^{1/2})^\vee$. With Corollary $\ref{fg39fh983hf8380f30j}$ we can then deduce: when $g\geq 5$, the first obstruction cohomology of the model $\big(\Scl\Mcl_g ,T^*_{\Scl\Mcl_g, -}\big)$ is non-vanishing. 
\end{EX}

\subsection{A Batchelor-Type Theorem}
We have so far been concerned with complex-analytic (i.e., holomorphic or algebraic) supermanifolds. These are supermanifolds whose structure sheaves are sheaves of complex-analytic functions. If we relax this condition to smooth functions we obtain the notion of a `smooth supermanifold' (c.f., Remark \ref{dkdjvjdkbvkbdvbebvkun}). A classical result in the category of smooth supermanifolds is \emph{Batchelor's theorem}, which originally appeared in \cite{BAT}. It states:

\begin{THM}\label{rhf894hg89hf3j09j3f}
Any smooth supermanifold is split.\qed
\end{THM}

\noindent
The subtlety in generalising Theorem \ref{rhf894hg89hf3j09j3f} to the complex-analytic category is that the splitting from Theorem \ref{rhf894hg89hf3j09j3f} need not be analytic. Now note that Theorem \ref{rhf894hg89hf3j09j3f} can be formulated as a statement about a class of supermanifolds. It leads therefore to the following definition in analogy with Definition \ref{reiugfgf783h983hf89h30} of projectable models.

\begin{DEF}
\emph{A model $(X, T^*_{X, -})$ is said to be \emph{split} if every supermanifold modelled on $(X, T^*_{X, -})$ is split.}
\end{DEF}
 
\noindent
Batchelor's theorem in Theorem \ref{rhf894hg89hf3j09j3f} can now be phrased: \emph{any smooth model $(X, T^*_{X, -})$ is split}.\footnote{Following Remark \ref{dkdjvjdkbvkbdvbebvkun}, a model $(X, T^*_{X, -})$ is smooth if $X$ is taken to be a smooth manifold and $T^*_{X, -}$ a smooth vector bundle.} Hence, results involving the deduction of splitness of models might be termed `Batchelor-type theorems'. An elementary such theorem is the following (for models now in the complex analytic category).

\begin{LEM}\label{rfg783fh93hf983h03}
Let $(X, T^*_{X, -})$ be a model and suppose its obstruction space vanishes, i.e., that $H^1\big(X, \Ac_{T^*_{X, -}}\big) = (0)$. Then $(X, T^*_{X, -})$ is split.\qed
\end{LEM}

\noindent
The assumptions in Lemma \ref{rfg783fh93hf983h03} are quite strong. However, Batchelor's theorem in Theorem \ref{rhf894hg89hf3j09j3f} can be deduced from Lemma \ref{rfg783fh93hf983h03} since sheaves of modules over smooth functions are fine, i.e., have acyclic sheaf cohomology. We conclude with a stronger Batchelor type theorem now, invoking the property of `goodness' of a model and its first obstruction cohomology.

\begin{THM}\label{fnvburbvuibuinenfneoifeoo}
Let $(X, T^*_{X, -})$ be a good model with ${\bf H}^1_\pt\big(X, \Ac_{T^*_{X, -}}\big) = (0)$. Then $(X, T^*_{X, -})$ is split.
\end{THM}

\begin{proof}
We will show that any supermanifold atlas modelled on $(X, T^*_{X, -})$ will be strongly split (see Definition \ref{ruieff98fh83f309jf3jf4rf}) from whence the present theorem will follow. Recall that for each $j$ we have a commutative diagram (c.f., \eqref{Prfg84gf848f7h4fh9}),
\begin{align}
\xymatrix{
& \mbox{\v H}^1\big(X, \mathrm G_{T^*_{X, -}, j+1}\big)\ar[dr]^{\om_*}&
\\
H^0\big(X, \Ac_{T^*_{X, -}, j}\big)\ar[ur]^\dt
\ar[rr]_{\pt^1_{j+1}} & & 
H^1\big(X, \Ac_{T^*_{X, -}, j+1}\big)
}
\label{kcndjvbjkbkvbdjvndkcnekl22222}
\end{align}
By Theorem \ref{dkcdnlnvdvbebvukenvl} we know that $\pt^1_{j+1} = 0$ for $j$ odd in \eqref{kcndjvbjkbkvbdjvndkcnekl22222}. Now a central result in \cite{BETTHIGHOBS} is in the characterisation of good models. It was found that: \emph{a model $(X, T^*_{X, -})$ is good if and only if $\dt$ in \eqref{kcndjvbjkbkvbdjvndkcnekl22222} is trivial for all $j$}. Hence if $(X, T^*_{X, -})$ is a good model we find by commutativity of \eqref{kcndjvbjkbkvbdjvndkcnekl22222} that  $\pt^1_{j+1} = 0$ for \emph{all $j$}. Hence that $\img\big\{\pt : H^1\big(X, \Ac_{T^*_{X, -}}[-1]\big)\ra H^1\big(X, \Ac_{T^*_{X, -}}\big)\big\} = (0)$. Assuming in addition that ${\bf H}^1_\pt\big(X, \Ac_{T^*_{X, -}}\big) = (0)$ we have the equality in \eqref{rnii4fi7fh8fh83hf0333330} giving therefore $\ker\big\{\pt : H^1\big(X, \Ac_{T^*_{X, -}}\big)\ra H^1\big(X, \Ac_{T^*_{X, -}}[1]\big)\big\}  = (0)$. This shows that the obstruction to splitting \emph{any} supermanifold atlas modelled on $(X, T^*_{X, -})$ must vanish. Hence this supermanifold atlas must be strongly split.
\end{proof}

\begin{REM}
\emph{In the case where $X$ is a Riemann surface, Corollary \ref{fjcdkbvjdbvjhbvrbvbi} shows that the conditions in Theorem \ref{fnvburbvuibuinenfneoifeoo} are equivalent to those in Lemma \ref{rfg783fh93hf983h03}.}
\end{REM}

\appendix
\numberwithin{equation}{subsection}

\section{Proof of Vanish-Lift-Vanish}
\label{fklnvkbvjrbvnfkemlkmel}

\subsection{Proof of Proposition $\ref{rfh89hf983hf893f93}$}
The proof of Proposition \ref{rfh89hf983hf893f93} is outlined in \cite{DW1} and, motivated by this argument, a particular case is addressed in \cite[Appendix A]{BETTHIGHOBS}. We continue this argument here.

\subsubsection{Group Actions on Sheaves}
We begin with the following definition of groups acting on sheaves of modules.

\begin{DEF}\label{ldlvnkjbvjbrvbuenveio}
\emph{Let $\Fc$ be a sheaf of $\Oc_X$-modules on $X$ and fix a group $G$. We say $G$ \emph{acts on $\Fc$} if:
\begin{enumerate}[$\bt$]
	\item each $g\in G$ defines an $\Oc_X$-module homomorphism $g\cdot :\Fc\ra \Fc$;
	\item the morphism defined by the identity element $e\in G$ is the identity morphism $e\cdot = {\bf 1}_\Fc: \Fc\ra\Fc$;
	\item for any two $g, h\in G$ the following diagram commutes,
	\begin{align}
	\xymatrix{
	\ar[dr]_{(gh)\cdot}\Fc \ar[r]^{h\cdot} & \Fc\ar[d]^{g\cdot}
	\\
	&\Fc
	}
	\label{rfh4hf984hf8hf0j09}
	\end{align}
\end{enumerate}}
\end{DEF}

\noindent
Evidently, if a group $G$ acts on a sheaf $\Fc$ then commutativity in \eqref{rfh4hf984hf8hf0j09} implies that the morphism defined by $g\in G$ will be an isomorphism. By functoriality on cohomology we obtain, for each $g\in G$ and integer $n$, an isomorphism on sheaf cohomology $g\cdot_*: H^n(X, \Fc) \stackrel{\cong}{\ra} H^n(X, \Fc)$ with $g\cdot_*h\cdot_* = (gh)\cdot_*$. We will ultimately be interested in the action induced on cohomology in degrees zero and one.

\subsubsection{The Dilation Action}
Suppose $\Fc$ is a sheaf of $\Zbb$-graded, $\Oc_X$-modules on a complex space $X$. Set 
\begin{align}
\Fc = \oplus_j \Fc_j.
\label{dkncdnvkjbvjkbr}
\end{align}
Then there exists a natural action of $\Cbb^\times$ on $X$ known as \emph{dilation}. It it defined as follows: firstly, any homogeneous $f\in \Fc$ defines a $\Zbb$-parity $p(f)\in \Zbb$, where $p$ is defined by sending $f$ to its $\Zbb$-grading. Clearly the parity map depends on the choice of $\Zbb$-grading on $\Fc$ in \eqref{dkncdnvkjbvjkbr}. With this choice of grading and parity map consider:
\begin{align}
\Cbb^\times \times \Fc\lra\Fc
&&
\mbox{given by}
&&
\big(\lam, f\big)
\longmapsto
\lam^{p(f)}f
\label{cnekbvurbviurvurnvor}
\end{align}
where $f$ is homogeneous. The mapping extends to inhomogeneous sections of $\Fc$ by linearity and defines an action of the group $\Cbb^\times$ on the sheaf $\Fc$ in the sense of Definition \ref{ldlvnkjbvjbrvbuenveio}. This action is known as the \emph{dilation action}.

\begin{REM}\label{bcuyegf87gf793hf983h839}
\emph{Any sheaf of $\Oc_X$-modules admits a dilation action as we can view it as being trivially $\Zbb$-graded. Hence the choice of $\Zbb$-grading is a crucial ingredient in forming the dilation action. We will suppress any mention of the choice of $\Zbb$-grading however if it is clear from the context.}
\end{REM}

\begin{EX}\label{rivrhhg04j09j0f}
Fix a locally free sheaf $T^*_{X, -}$ on $X$. Then the exterior algebra $\Fc = \wedge^\bt T^*_{X, -}$ carries a natural $\Zbb$-grading $\wedge^\bt T^*_{X, -} = \oplus_j \wedge^jT^*_{X, -}$ into exterior powers and hence a natural dilation action as in \eqref{cnekbvurbviurvurnvor}. On homogeneous components we have,
\begin{align*}
\Cbb^\times\times \wedge^j T^*_{X, -} 
\lra 
\wedge^j T^*_{X, -}
&&
\big(\lam, u\big) \longmapsto \lam^j u
\end{align*}
for all $j$.
\end{EX}

\noindent
We note that the dilation action in \eqref{cnekbvurbviurvurnvor} extends to an action on tensor products. If $\Fc$ and $\Fc^\p$ are $\Zbb$-graded the so is their tensor product. Let $p$ and $p^\p$ denote the respective parity maps on $\Fc$ and $\Fc^\p$. Then $p+p^\p$ is the parity map on $\Fc\otimes \Fc^\p$ and so the dilation action on the tensor product is:
\begin{align}
\Cbb\times \Fc\otimes \Fc^\p
\lra 
\Fc\otimes\Fc^\p
&&
\mbox{given by}
&&
\big(\lam, f\otimes f^\p\big)
\longmapsto \lam^{p(f) + p^\p(f^\p)} f\otimes f^\p.
\label{jdkbkjbvoinvioenviep}
\end{align}
If $\Fc$ is $\Zbb$-graded as in \eqref{dkncdnvkjbvjkbr}, then we consider its dual be inversely graded to $\Fc$, i.e., that 
\begin{align}
\Fc^*  = \oplus_j \Fc^*_{-j}.
\label{kfvbjkdbvkjrbvuirbvuenoiepe}
\end{align}
Hence if $p$ is the parity map for $\Fc$ then $-p$ is the parity map for its $\Oc_X$-dual $\Fc^*$. Its endomorphisms therefore dilate as follows:
\begin{align}
\Cbb^\times\times \mathcal End_{\Oc_X}\Fc
\lra \mathcal End_{\Oc_X}\Fc
&&
\mbox{with}
&&
\big(\lam, f\otimes f^*\big)
\longmapsto 
\lam^{p(f) - p(f^*)}f\otimes f^*.
\label{fbciebcgcg7c9839h}
\end{align}
With this observation we continue Example \ref{rivrhhg04j09j0f}.

\begin{EX}\label{djckbvdkfdbvbbcnc}
The $j$-th exterior power $\wedge^j T^*_{X,-}$ is the $j$-th graded component of $\wedge^\bt T^*_{X, -}$. Dualising, we see that $\wedge^j T_{X, -}$ is the $j$-th graded component of $(\wedge^\bt T^*_{X, -})^*$. Evidently, by \eqref{fbciebcgcg7c9839h} we find for each $j$ and $k$, the dilation:
\begin{align*}
\Cbb^\times 
\times 
\wedge^j T^*_{X, -}\otimes \wedge^k T_{X, -}
\lra 
\wedge^j T^*_{X, -}\otimes \wedge^k T_{X, -}
&&
\mbox{given by}
&&
\big(\lam, u\otimes v\big)
\longmapsto
\lam^{j-k}u\otimes v.
\end{align*}
Evidently, the endomorphisms $\mathcal End_{\Oc_X}T^*_{X, -}$ are invariant under dilation.
\end{EX}

\subsubsection{On Obstruction Sheaves}
Green in \cite{GREEN} derived the following important characterisation of the obstruction sheaves. For $\Ac_{T^*_{X, -}}$ the obstruction sheaf of the model $(X, T^*_{X -})$ and $\Ac_{T^*_{X, -}, j}$ its $j$-th graded component there exists an isomorphism
\[
\Ac_{T^*_{X, -}, j}
\cong 
\left\{
\begin{array}{ll}
\wedge^j T^*_{X, -}\otimes T_X
&
\mbox{if $j$ is even}
\\
\wedge^j T^*_{X, -}\otimes T^*_{X, -}
&
\mbox{if $j$ is odd}.
\end{array}
\right.
\]
In viewing $T_X$ as trivially $\Zbb$-graded (c.f., Remark \ref{bcuyegf87gf793hf983h839}) we obtain the following action on the obstruction sheaf from \eqref{jdkbkjbvoinvioenviep} and Example \ref{djckbvdkfdbvbbcnc}.


\begin{LEM}\label{rfg784gf9hf03j099j33333}
For any $j$ we have the action
\begin{align*}
\Cbb^\times \times  \Ac_{T^*_{X, -}, j} \lra \Ac_{T^*_{X, -}, j}
&&
(\lam, w) \longmapsto 
\left\{
\begin{array}{ll}
\lam^jw&\mbox{if $j$ is even}
\\
\lam^{j-1}w&\mbox{if $j$ is odd}
\end{array}
\right.
\end{align*}
\qed
\end{LEM}

\noindent
The action in Lemma \ref{rfg784gf9hf03j099j33333} gives an action on global sections. We wish to compare this latter action with that on $\mbox{\v H}^1\big(X, \mathrm G_{T^*_{X, -}, j+1}\big)$. Our objective is thus to prove the following:

\begin{PROP}\label{fmdvbmdbvbmvbekjve}
For any $j$ there exists an action of $\Cbb^\times$ on $\mbox{\emph{\v H}}^1\big(X, \mathrm G_{T^*_{X, -}, j+1}\big)$ commuting the following diagram,
\[
\xymatrix{
\ar[d]_{1\times\dt} \Cbb^\times \times H^0\big(X, \Ac_{T^*_{X, -}, j}\big) \ar[rr] & &  H^0\big(X, \Ac_{T^*_{X, -}, j}\big)\ar[d]^{\dt}
\\
\Cbb^\times \times \mbox{\emph{\v H}}^1\big(X, \mathrm G_{T^*_{X, -}, {j+1}}\big)\ar[d]_{1\times \om_*} \ar[rr] & & \mbox{\emph{\v H}}^1\big(X, \mathrm G_{T^*_{X, -}, {j+1}}\big) \ar[d]^{\om_*}
\\
\Cbb^\times \times H^1\big(X, \Ac_{T^*_{X, -}, j+1}\big) \ar[rr] & &  H^1\big(X, \Ac_{T^*_{X, -}, j+1}\big)
}
\]
\end{PROP}

\begin{proof}
We will firstly show that $\Cbb^\times$ acts on the sheaf of groups $\mathrm G_{T^*_{X, -}, k}$ for any $k$. We recall the definition of $\mathrm G_{T^*_{X, -}, k}$ from \eqref{fvnrnvnkjnkfjnwjkfnjkfnk} below for convenience,
\begin{align}
\mathrm G_{T^*_{X, -}, k}
=
\big\{
\al\in \mathcal Aut_{\Zbb_2}\wedge^\bt T^*_{X, -}\mid \al(u) - u\in J^k,~\forall u\in \wedge^\bt T^*_{X, -}\big\}
\end{align}
where $J\subset \wedge^\bt T^*_{X, -}$ is the irrelevant ideal. Now from Example \ref{rivrhhg04j09j0f} we see how $\Cbb^\times$ acts on the exterior algebra $\wedge^\bt T^*_{X, -}$. This action clearly preserves the ideal $J$. Moreover, for any $\al\in \mathrm G_{T^*_{X, -},k}$, since $\al\equiv {\bf 1}$ modulo $J^k$ it follows that $\lam\al\lam^{-1}\equiv \lam{\bf 1}\lam^{-1} = {\bf 1}$ modulo $J^k$ and for any $\lam\in\Cbb^\times$. Hence that $\lam\al\lam^{-1}\in \mathrm G_{T^*_{X, -}, k}$ for any $\lam\in\Cbb^\times$. As can be checked, the conjugation $\al\mapsto\lam\al\lam^{-1}$ defines a group action $\Cbb^\times$ on $\mathrm G_{T^*_{X, -}, k}$ for each $k$. Now more explicitly, since $\al(u) - u\in J^k$ we can write $\al = {\bf 1} + D + \ldots$, where $D : \wedge^\bt T^*_{X, -}\ra \wedge^\bt T^*_{X, -}[k]$ and the ellipses contain terms sending $\wedge^\bt T^*_{X, -}\ra \wedge^\bt T^*_{X, -}[\ell]$ for $\ell> k$. Modulo $J^{k+1}$, the term $D$ will define a derivation. Now suppose $k$ is even. Then for any $u\in \wedge^\bt T^*_{X,-}$ the $0$-th graded component maps to an element in $\wedge^kT^*_{X, -}$. Therefore, modulo $J^{k+1}$ we find
\begin{align}
(\lam\al\lam^{-1})(u)
=
\lam\big(\al(\lam^{-1} u)\big)
&=
\lam\big({\bf 1}(\lam^{-1} u) + D(\lam^{-1} u) + \cdots\big)
\notag
\\
&=
u + \lam^k D(u) \mod J^{k+1}
\label{rf847gf78gf97h398fh3}
\end{align}
If $k$ is odd the restriction of $\al$ to $\wedge^0T^*_{X, -} = \Oc_X$ is trivial since it preserves the $\Zbb_2$-grading $\wedge^\bt T^*_{X, -}$. Hence for any $u\in \wedge^\bt T^*_{X, -}$ we find
\begin{align}
(\lam\al\lam^{-1})(u)
&= 
\lam\big({\bf 1}(\lam^{-1} u) + D(\lam^{-1} u) + \cdots\big)
\notag
\\
&=
u +  \lam^{k-1} D(u) \mod J^{k+1}
\label{dvbdbvdbvjknee}
\end{align}
Comparing \eqref{rf847gf78gf97h398fh3} and \eqref{dvbdbvdbvjknee} with the action in Lemma \ref{rfg784gf9hf03j099j33333} we see that the projection $ \mathrm G_{T^*_{X, -}, k}\ra  \mathrm G_{T^*_{X, -}, k}/ \mathrm G_{T^*_{X, -}, k+1} = \Ac_{T^*_{X, -}, k}$ will be $\Cbb^\times$-equivariant, i.e., that the following diagram will commute 
\[
\xymatrix{
\Cbb^\times \times\mathrm G_{T^*_{X, -}, k} \ar[d] \ar[r]  & \mathrm G_{T^*_{X, -}, k}\ar[d]
\\
\Cbb^\times \times \Ac_{T^*_{X, -}, k} \ar[r] &  \Ac_{T^*_{X, -}, k}.
}
\]
Thus for each $\lam\in \Cbb^\times$ we get an isomorphism of short exact sequences of sheaves,
\begin{align}
\xymatrix{
\{e\}\ar[r] & \mathrm G_{T^*_{X, -}, k+1} \ar[d]_{\lam\cdot} \ar[r] & \mathrm G_{T^*_{X, -}, k} \ar[d]_{\lam\cdot} \ar[r] & \Ac_{T^*_{X, -}, k}\ar[d]_{\lam\cdot} \ar[r] & \{e\}
\\
\{e\}\ar[r] & \mathrm G_{T^*_{X, -}, k+1} \ar[r] & \mathrm G_{T^*_{X, -}, k} \ar[r] & \Ac_{T^*_{X, -}, k}\ar[r] & \{e\}
}
\label{fdvjbrjvbbvkjenjkvnekj}
\end{align}
giving then on cohomology the following commutative diagrams,
\begin{align*}
\xymatrix{
\ar[d]_\dt H^0\big(X, \Ac_{T^*_{X, -}, k}\big) \ar[r]^{\lam_*}
& H^0\big(X, \Ac_{T^*_{X, -}, k}\big) \ar[d]_\dt
\\ 
\mbox{\v H}^1\big(X, \mathrm G_{T^*_{X, -}, k+1}\big)\ar[r]^{\lam_*}
&
\mbox{\v H}^1\big(X, \mathrm G_{T^*_{X, -}, k+1}\big).
}
&&
\xymatrix{
\ar[d]_{\om_{k*}} \mbox{\v H}^1\big(X, \mathrm G_{T^*_{X, -}, k}\big) \ar[r]^{\lam_*}
& \mbox{\v H}^1\big(X, \mathrm G_{T^*_{X, -}, k}\big) \ar[d]_{\om_{k*}}
\\ 
H^1\big(X, \Ac_{T^*_{X, -}, k}\big) \ar[r]^{\lam_*}
&
H^1\big(X, \Ac_{T^*_{X, -}, k}\big).
}
\end{align*}
Commutativity of the above diagrams for each $\lam\in \Cbb^\times$ is precisely the statement in this proposition.
\end{proof}

\noindent
Proposition \ref{rfh89hf983hf893f93} will now follow from Lemma \ref{rfg784gf9hf03j099j33333} and Proposition \ref{fmdvbmdbvbmvbekjve} as follows. Firstly, recall that the boundary map $\pt : H^0\big(X, \Ac_{T^*_{X, -}, j}\big) \ra H^1\big(X, \Ac_{T^*_{X, -}, j+1}\big)$ was defined by reference to $\mbox{\v H}^1\big(X, \mathrm G_{T^*_{X, -}, j}\big)$. That is, the following diagram commutes for each $j$ (c.f., \eqref{Prfg84gf848f7h4fh9}),
\begin{align}
\xymatrix{
& \mbox{\v H}^1\big(X, \mathrm G_{T^*_{X, -}, j+1}\big)\ar[dr]^{\om_*}&
\\
H^0\big(X, \Ac_{T^*_{X, -}, j}\big)\ar[ur]^\dt
\ar[rr]_{\pt^1_{j+1}} & & 
H^1\big(X, \Ac_{T^*_{X, -}, j+1}\big)
}
\label{kcndjvbjkbkvbdjvndkcnekl}
\end{align}
Now suppose $j = 2\ell+1$ is odd. Since $\pt^1$ is linear we have by Lemma \ref{rfg784gf9hf03j099j33333} and any $w\in H^0\big(X, \Ac_{T^*_{X, -}, 2\ell+1}\big)$,
\begin{align}
\pt^1_{j+1} (\lam\cdot w) 
= 
\pt^1_{j+1}(\lam^{2\ell}w)
=
\lam^{2\ell}\pt_{j+1}^1(w).
\label{rfuhf98hf98h3f30j0}
\end{align}
But now consider that the diagram in \eqref{kcndjvbjkbkvbdjvndkcnekl} commutes. Therefore, by Lemma \ref{rfg784gf9hf03j099j33333} and Proposition \ref{fmdvbmdbvbmvbekjve} we find
\begin{align}
\pt_{j+1}^1(\lam\cdot w) = \om_*(\dt(\lam\cdot w)) = \om_*(\lam \star\dt(w)) = \lam^{2\ell+2}\om_*\dt(w) = \lam^{2\ell+2}\pt_{j+1}^1(w).
\label{rcnibviubiuvnjneoe}
\end{align}
This is compatible with in \eqref{rfuhf98hf98h3f30j0} if and only if $\pt_{j+1}^1 = 0$ for all $j = 2\ell+1$, odd. Hence the odd component of $\pt$ must vanish.\qed

\begin{REM}
\emph{The dilation action is the centrepiece of the proof of Proposition \ref{rfh89hf983hf893f93} above. By Remark \ref{bcuyegf87gf793hf983h839}, the formation of the dilation action involves a choice of grading and hence appears on a first glance to be arbitrary. We emphasise that while this might be so, what is important is the deduction of the commutative diagram in Proposition \ref{fmdvbmdbvbmvbekjve} which meaningfully incorporates our construed dilation action.}
\end{REM}

\subsection{Proof of Proposition $\ref{rgf78f79hf93h98fh39h93}$}
The argument follows along similar lines to Proposition \ref{rfh89hf983hf893f93}. Firstly recall that the composition $\pt^1_{j+1}$ in \eqref{kcndjvbjkbkvbdjvndkcnekl} is linear for all $j$. Hence the composition $\widetilde\pt$ in \eqref{rfh794fh97hf983hf93h} will be linear. We can now make an analogous comparison as in \eqref{rfuhf98hf98h3f30j0} and \eqref{rcnibviubiuvnjneoe}. If $w\in H^0\big(X, \Ac_{T^*_{X, -}, 2\ell+1}\big)$ and $\lam\in \Cbb^\times$ then
\[
\widetilde\pt(\lam\cdot w) = \widetilde\pt (\lam^{2\ell}w) = \lam^{2\ell}\widetilde\pt (w).
\]
However, we also have:
\[
\widetilde\pt(\lam\cdot w)
=
\om_*\big(\widetilde\dt_{2\ell+2}(\lam\cdot w)\big)
=
\om_*\big(\lam\star \widetilde\dt_{2\ell+2}(w)\big)
=
\lam^{2\ell+2}\widetilde\pt(w).
\]
Hence $\widetilde\pt = 0$.\qed

\subsection{Proof of Theorem $\ref{dkcdnlnvdvbebvukenvl}$} 
In Proposition \ref{rgf78f79hf93h98fh39h93} we found that the linearisation of the lift $\widetilde\dt_{2\ell+1}$ of $\dt_{2\ell+1}$ vanished. Hence we will obtain a further lift of $\widetilde\dt_{2\ell+1}$, represented below by the dashed arrow:
\[
\xymatrix{
& &\mbox{\v H}^1\big(X, \mathrm G_{T^*_{X, -}, 2\ell+4}\big)\ar[d]
\\
\ar[drr]_{\dt_{2\ell+2}} \ar@{-->}[urr]^{\widetilde\dt_{2\ell+2}^\p} H^0\big(X, \Ac_{T^*_{X, -}, 2\ell+1}\big) \ar[rr]|{\widetilde\dt_{2\ell+2}} & &
\mbox{\v H}^1\big(X \mathrm G_{T^*_{X, -}, 2\ell+3}\big)\ar[d]
\\
& & \mbox{\v H}^1\big(X, \mathrm G_{T^*_{X,-}, 2\ell+2}\big)
}
\]
Arguing as in Proposition \ref{rgf78f79hf93h98fh39h93} the linearisation of $\widetilde\dt_{2\ell+2}^\p$ will vanish, leading therefore to a further lift. In this way we see that the map $\dt_{2\ell+2}: H^0\big(X, \Ac_{T^*_{X, -}, 2\ell+1}\big) \ra \mbox{\v H}^1\big(X, \mathrm G_{T^*_{X,-}, 2\ell+2}\big)$ will lift to $\widetilde\dt^{(k)} : H^0\big(X, \Ac_{T^*_{X, -}}\big) \ra  \mbox{\v H}^1\big(X, \mathrm G_{T^*_{X,-}, k}\big)$ for any $k> 2\ell+2$. Since $\mathrm G_{T^*_{X,-}, k} = (e)$ is trivial for sufficiently large $k$ we obtain the following commutative diagram
\[
\xymatrix{
& &\{e\}\ar[d]
\\
\ar[urr] H^0\big(X,  \mathrm G_{T^*_{X,-}, 2\ell+2}\big) \ar[rr]_{\dt_{2\ell+2}} && \mbox{\v H}^1\big(X, \mathrm G_{T^*_{X, -}, 2\ell+2}\big)
}
\]
which therefore shows that $\dt_{2\ell+2}$ is trivial, as required.
\qed

\bibliographystyle{alpha}
\bibliography{Bibliography}

\hfill
\\
\noindent
\small
\textsc{
Kowshik Bettadapura 
\\
\emph{Yau Mathematical Sciences Center} 
\\
Tsinghua University
\\
Beijing, 100084, China}
\\
\emph{E-mail address:} \href{mailto:kowshik@mail.tsinghua.edu.cn}{kowshik@mail.tsinghua.edu.cn}

\end{document}